\documentclass{amsart}
\usepackage{amsfonts,amssymb,amsmath,amsthm}
\usepackage{url}
\usepackage{enumerate}

\newtheorem{theorem}{Theorem}[section]
\newtheorem{example}[theorem]{Example}
\newtheorem{lemma}[theorem]{Lemma}
\newtheorem{proposition}[theorem]{Proposition}
\newtheorem{corollary}[theorem]{Corollary}
\theoremstyle{definition}
\newtheorem{definition}[theorem]{Definition}
\newtheorem{remark}[theorem]{Remark}
\numberwithin{equation}{section}

\author{Hiroshi~Nozaki}
\address{
        Department of Mathematics \\
        Aichi University of Education \\
        Igaya-cho, Kariya-city 448-8542 \\
	Japan}
\email{hnozaki@auecc.aichi-edu.ac.jp}

\author{Masanori~Sawa}
\address{
Graduate School of Information Sciences \\
	Nagoya University \\
	Chikusa-ku, Nagoya 464-8601 \\
	Japan}
\email{sawa@is.nagoya-u.ac.jp}

\thanks{
The second author is supported in part by
Grant-in-Aid for Young Scientists (B) 22740062 and
Grant-in-Aid for Challenging Exploratory Research 23654031
by the Japan Society for the Promotion of Science}

\begin{document}

\title[
Remarks on Hilbert identities, embeddings, and invariant cubature
]{Remarks on Hilbert identities, isometric embeddings, and invariant cubature}

\begin{abstract}
Victoir (2004) developed a method to construct cubature formulae
with various combinatorial objects.
Motivated by this,
we generalize Victoir's method with one more combinatorial object,
called regular $t$-wise balanced designs.
Many cubature of small indices with few points are provided,
which are used to update Shatalov's table (2001) of isometric embeddings
in small-dimensional Banach spaces, as well as to improve
some classical Hilbert identities.
A famous theorem of Bajnok (2007)
on Euclidean designs invariant under the Weyl group of Lie type $B$
is extended to all finite irreducible reflection groups.
A short proof of the Bajnok theorem is presented
in terms of Hilbert identities.
\end{abstract}

\maketitle

\section{Introduction} \label{sect1}

Let $p$ be a positive integer such that $p \ne \infty$.
The $m$-dimensional Euclidean space $\mathbb{R}^m$ is
a Banach space $l_p^m$ endowed with the norm
$$
\|x \|_p =
\bigg( \sum_{i=1}^m |x_i|^p \bigg)^{1/p}.
$$
Given two spaces $l_p^m$ and $l_q^n$,
a classical problem in Banach space theory asks when
there is an $\mathbb{R}$-linear map $F : l_p^m \rightarrow l_q^n$ such that
$$
\| F(x) \|_q = \| x \|_p
$$
for every $x \in l_p^m$.
Such a map
is called an {\it isometric embedding from $l_p^m$ to $l_q^n$}.
To exclude trivial cases, we assume that
$n \ge m \ge 2$ and $p \ne q$.
It is known~\cite[Theorem 1.1]{LV93} that
if $p, q \ne \infty$ and
an isometric embedding from $l_p^m$ to $l_q^n$ exists, then
$p = 2$ and $q$ is an even integer.
Throughout this paper we only consider the case where
$p = 2$ and $q$ is even, and
fix the notations $p, q, m, n$.

Isometric embeddings are closely related to
a representation of $(\sum_{i=1}^m x_i^2)^{q/2}$ as
a sum of $q$th powers of linear forms with positive real coefficients.
Such representations originally stem from
a work of Hilbert on Waring's problem~\cite{H09},
and therefore called {\it Hilbert identities}~\cite{R11}.
Hilbert
solved Waring's problem, showing on the way
that there exist isometric embeddings
$l_2^m \rightarrow l_q^n$
with $n$ depending on
$m$ and $q$.
Several alternative proofs of Hilbert's theorem are known;
for example, see~\cite{D23},~\cite{E71} and the references therein.
But most of them, including the original by Hilbert,
involve non-constructive arguments in analysis, and
do not give any explicit constructions of embeddings
\footnote{
Bruce Reznick kindly told us that
Stridsberg's proof (1912) is constructive,
if we know how to compute the roots of Hermite polynomials.}.
Thus publications with explicit embeddings continued to appear.

Isometric embeddings are also related to
a certain object in numerical analysis.
Let $\Omega$ be a subset of $\mathbb{R}^m$
on which a normalized measure $\mu$ is defined.
A finite subset $X$ of $\Omega$ with a positive weight $w$
is called a {\it cubature formula of index $q$} if
\begin{align}
\label{eq:CF}
\int_\Omega f(x) \mu ({\rm d} x)
=
\sum_{x \in X} w(x) f(x)
\end{align}
for every $f \in {\rm Hom}_q(\Omega)$, where ${\rm Hom}_q(\Omega)$
is the space of all homogeneous polynomials of degree $q$
restricted to $\Omega$.
Lyubich and Vaserstein~\cite{LV93} and Reznick~\cite{R92}
proved the equivalence between
an embedding $l_2^m \rightarrow l_q^n$ and
an $n$-point cubature of index $q$
for the surface measure $\rho$
on the $(m-1)$-dimensional unit sphere $S^{m-1}$.

Many papers are devoted to the construction of
spherical cubature formulae.
There are two classical approaches:
One uses orbits of finite subgroups of the orthogonal group $O(m)$
acting on $S^{m-1}$~\cite{S62},
and the other takes ^^ ^^ product" of
several lower-dimensional cubature~\cite{S71}.
Cubature formulae that are studied
in the context of numerical analysis and related areas,
are often of degree type.
Victoir~\cite{V04} developed a novel technique
to construct degree-type cubature
for integrals with special symmetry.
His idea is as follows:
Given a cubature formula invariant under
the Weyl group of Lie type $B$,
one eliminates some specified points of the formula
by using combinatorial objects such as
$t$-designs and orthogonal arrays.
With this method,
Victoir found many cubature of small degrees
with few points in general dimensional spaces.

This paper has several important aims.
First, we generalize the Victoir method with
a special class of block designs,
called {\it regular $t$-wise balanced designs}.
The concept of regular $t$-wise balanced designs has been substantiated
by applications in statistics~\cite{FKJ89,GJ83,KM90},
however, it seems that
there is insufficient evidence to support it
from other mathematical aspects.
To find a new meaning of this concept,
as well as to let it know researchers in many areas of mathematics
are both important aims of this paper.
On the other hand,
Bajnok~\cite[Theorem 3]{B07} proved that
Euclidean
designs, a generalization of spherical cubature,
that are invariant under the Weyl group of Lie type $B$
have degree at most $7$.
We further discuss the Bajnok theorem both from a combinatorial and analytic point of view.

This paper is organized as follows.
In Section 2 we review some basic facts and notions, and
explain the Victoir method in detail.
In Section 3 we generalize the Victoir method with regular $t$-wise
balanced designs.
In Section 4,
we give general-dimensional index-four and -six cubature,
together with some extra examples of index-six cubature
that improve Shatalov's table~\cite[Theorem 4.7.20]{S01} of
isometric 
embeddings
$l_2^m \rightarrow l_6^n$.
In Section 5,
we generalize the Bajnok theorem for all finite irreducible reflection groups,
and thereby classify spherical cubature
with a certain geometric meaning.
In Section 6,
some of the cubature constructed in Sections 4 and 5
are translated into Hilbert identities,
in order to improve classical identities
as such by Schur~\cite{D23} and Reznick~\cite{R92}.
An extremely short proof of the Bajnok theorem is given
in terms of Hilbert identities.

\section{Preliminaries}
\label{sect2}

\subsection{Isometric embeddings and Hilbert identities}

Lyubich and Vaserstein~\cite{LV93} and Reznick~\cite{R92}
observed a close relationship between Hilbert identities,
isometric embeddings, and spherical cubature formulae.
\begin{theorem}
\label{thm:LV93}
The following are equivalent. \\
(i) There exists a cubature formula of index $q$ on $S^{m-1}$ with $n$ points; \\
(ii) There exists an isometric embedding $l_2^m \rightarrow l_q^n$; \\
(iii) There exist $n$ vectors $r_1, \ldots, r_n \in \mathbb{R}^m$ such that
for any $x \in \mathbb{R}^m$
\begin{equation*}
\langle x, x \rangle^\frac{q}{2}
= \sum_{i=1}^n \langle x, r_i \rangle^q.
\end{equation*}
\end{theorem}

We explain Theorem~\ref{thm:LV93} in detail
for further arguments in the following sections.
Assume that
points $x_1, \ldots, x_n \in S^{m-1}$ and weights $w_1, \ldots, w_n$
form a cubature of index $q$ on $S^{m-1}$.
Let $\langle x, y \rangle^q \in {\rm Hom}_q(\mathbb{R}^m)$,
where
$\langle \cdot, \cdot \rangle$ denotes the usual inner product.
Then
$$
\sum_{i=1}^n w_i \langle x, x_i \rangle^q
=
\int_{S^{m-1}} \langle x, y \rangle^q \rho ({\rm d} y)
=
\langle x, x \rangle^\frac{q}{2}
c_q,
$$
where
\[
c_q = \int_{S^{m-1}} y_1^q \rho ({\rm d} y),\quad y=(y_1, \ldots, y_m).
\]
This is, equivalently,
$$
\langle x, x \rangle^\frac{q}{2}
=
\sum_{i=1}^n \langle x, r_i \rangle^q,
$$
where $r_i = \sqrt[q]{w_i/c_q} x_i$.
This polynomial identity is further transformed as follows:
$$
\langle x, x \rangle^\frac{1}{2}
=
\bigg(
\sum_{i=1}^n \langle x, r_i \rangle^q \bigg)^\frac{1}{q},
$$
which implies that the mapping
$$
x \mapsto
(\langle x, r_1 \rangle, \ldots, \langle x, r_n \rangle )
$$
is an isometric embedding $l_2^m \rightarrow l_q^n$.

By the early fundamental works of
Hilbert~\cite{H09},
there is a positive integer $N(m, q)$ such that
for any $n \ge N(m, q)$
an isometric embedding $l_2^m \rightarrow l_q^n$ exists.
It is known (cf.~\cite{R92}) that
\begin{equation}
\label{eq:bound}
\binom{m+\tfrac{q}{2}-1}{m-1} \le
N(m, q) \le \binom{m+q-1}{m-1}.
\end{equation}
The lower- and upper-bound part of (\ref{eq:bound}) mean
the dimension of ${\rm Hom}_{q/2}(\mathbb{R}^m)$
and ${\rm Hom}_q(\mathbb{R}^m)$ respectively.

\subsection{Cubature formulae}

Let
$\Omega \subset \mathbb{R}^m$, and
$\mu$ be a normalized measure on $\Omega$
such that $\Omega, \mu$ are both
invariant under the group $O(m)$.
We assume that
polynomials are integrable up to sufficiently large degrees for
$$
\mathcal{I}[f] =
\int_\Omega f(x) \mu ({\rm d} x).
$$

Let $X$ be a finite set in $\mathbb{R}^m$ with a positive weight $w$.
The pair $(X, w)$
is called
a {\it cubature formula of degree $q$ for $\mathcal{I}$}
if
$$
\mathcal{I}[f] = \sum_{x \in X} w(x) f(x)
$$
for every $f \in \mathcal{P}_q(\Omega)$, where
$\mathcal{P}_q(\Omega)$ denotes the space of all polynomials of degree at most $q$
restricted to $\Omega$.
In particular,
a spherical cubature is called a {\it spherical design}
if $w$ is a constant weight.

A subset $X$ of $\mathbb{R}^m$ is said to be {\it antipodal} if
it is partitioned into $\tilde{X}, -\tilde{X}$, namely,
$X = \tilde{X} \cup (-\tilde{X})$ and $\tilde{X} \cap (-\tilde{X}) = \emptyset$.
A cubature formula $(X, w)$ is {\it centrally symmetric} if
$X$ is antipodal and $w(x) = w(-x)$ for any $x \in X$.
The following mentions the relationship among
degree-type and index-type spherical cubature.

\begin{proposition}
\label{prop:antipodality}
{\rm (\cite[Proposition 4.3]{LV93})}. 
Let $X$ be an
antipodal finite subset of $S^{m-1}$.
Then
$X$ is a centrally symmetric cubature formula
on $S^{m-1}$ of degree $q+1$ with $2n$ points iff
$\tilde{X}$ is a cubature formula on $S^{m-1}$ of index $q$ with $n$ points.
\end{proposition}

We are interested in the following type of integrals:
\begin{equation}
\label{eq:SSI}
\int_{\mathbb{R}^m} f(x)
W(\| x \|_2) {\rm d} x,
\end{equation}
where $W$ is a density function on $\mathbb{R}^m$.
Such integrals are often considered in the context of analysis;
for example see~\cite{X00-2}.
\begin{proposition}
\label{prop:SSI}
If points $x_1, \ldots, x_n$ and weights
$w_1, \ldots, w_n$ form a cubature formula of index $q$
for (\ref{eq:SSI}),
then the points
$x_i/ \|x_i \|_2$
and the weights
$\frac{\|x_i \|_2^q w_i}{\int_0^\infty r^{q+m-1} W(r) {\rm d} r}$
form a cubature formula of index $q$ on $S^{m-1}$.
Conversely, if $x_1, \ldots, x_n$ and
$w_1, \ldots, w_n$ form a cubature formula of index $q$ on $S^{m-1}$, then
the points $x_i$ and weights $w_i \int_0^\infty r^{q+m-1} W(r) {\rm d} r$ form
a cubature formula of index $q$ for (\ref{eq:SSI}).
\end{proposition}
\begin{proof}
The result follows by observing that
for any $f \in {\rm Hom}_q(\mathbb{R}^m)$,
\begin{align*}
\int_{\mathbb{R}^m} f(x) W(\| x \|_2) {\rm d} x
&=
\int_0^\infty \Big( \int_{S^{m-1}} f(r x)
\rho ({\rm d} x) \Big) r^{m-1} W(r) {\rm d} r \\
& =
\int_0^\infty r^{q+m-1} W(r) {\rm d} r
\int_{S^{m-1}} f(x) \rho ({\rm d} x).
\end{align*}
\end{proof}

\begin{remark}
By Proposition~\ref{prop:SSI},
in order to construct spherical cubature,
we may find cubature for any integral of the form (\ref{eq:SSI}).
For example, one may think of Gaussian integrals.
Such cubature formulae are of particular interest
in probability theory~\cite{LV04} and algebraic combinatorics~\cite{BB04}.
Moreover,
the $m$-dimensional Gaussian integral can be represented simply
as the $m$-fold product of one-dimensional Gaussian integrals, which
is convenient for explaining Victoir's method.
\end{remark}

The following proposition is often used in Sections 3 and 4.
\begin{proposition}
\label{prop:SX1}
Let $X$ be an antipodal finite subset of $R^m$.
Let $\tilde{w}, w$ be weight functions on $\tilde{X}, X$, respectively,
such that for any $x \in \tilde{X}$,
$w(x) = w(-x) = 2\tilde{w}(x)$.
Then
$(\tilde{X}, \tilde{w})$ is a cubature formula of index $q$ for (\ref{eq:SSI})
if and only if
$(X, w)$ is a centrally symmetric cubature of
index $q$ for (\ref{eq:SSI}).
\end{proposition}

\subsection{The Sobolev theorem}
Let $G$ be a finite subgroup of $O(m)$, and
$f \in \mathcal{P}_t(\mathbb{R}^m)$.
We define the action of $\sigma \in G$ on $f$ as follows:
\begin{align*}
(\sigma f) (x) = f(x^{\sigma^{-1}}), \qquad x \in \mathbb{R}^m.
\end{align*}
A polynomial $f$ is said to be {\it $G$-invariant} if
$\sigma f = f$
for every $ \sigma \in G$.
We denote the set of $G$-invariant polynomials in
$\mathcal{P}_t(\mathbb{R}^m)$, ${\rm Harm}_t(\mathbb{R}^m)$
by $\mathcal{P}_t(\mathbb{R}^m)^G$, ${\rm Harm}_t(\mathbb{R}^m)^G$
respectively, where
$\text{Harm}_t(\mathbb{R}^m)$ is
the subspace of $\mathcal{P}_t(\mathbb{R}^m)$ of
harmonic homogeneous polynomials of degree $t$.

A cubature formula is said to be {\it $G$-invariant} if
the domain and measure of the integral
are invariant under $G$,
the points are a union of $G$-orbits $z_1^G, \ldots, z_e^G$, and
$w(x) = w(x')$
for any $x, x' \in z_i^G$;
the orbits $z_1^G, \ldots, z_e^G$ and weights $w_1, \ldots, w_e$
are said to {\it generate} the formula.

\begin{theorem}
[\cite{S62}]
\label{thm:Sobolev}
With the above set up,
a $G$-invariant cubature formula is of degree $t$ if and only if
it is exact for every polynomial
$f \in \mathcal{P}_t(\mathbb{R}^m)^G$.
\end{theorem}

Theorem~\ref{thm:Sobolev} is known as the Sobolev theorem,
which is at the core of the Victoir method, as seen in the next subsection.

The concept of Euclidean designs was
introduced by Neumaier and Seidel~\cite{NS88}
as a generalization of spherical cubature.
Let $X$ be a finite set in $\mathbb{R}^m$, and
$\{\| x \|_2 \mid x \in X\} = \{r_1, \cdots, r_p\}$.
Let $S_i^{m-1}$
be the sphere of radius $r_i$ centered at the origin,
and $X_i = X \cap S_i^{m-1}$.
To each $S_i$ the surface measure $\rho_i$ is assigned.
Let $|S_i^{m-1}| = \int_{S_i^{m-1}}\rho_i({\rm d}x)$,
where
$\frac{1}{|S_i^{m-1}|} \int_{S_i^{m-1}}f(x)
\rho_i({\rm d}x) = f(0)$ if $S_i^{m-1} = \{ 0 \}$.
\begin{definition}
[\cite{NS88}]
\label{def:Euclid}
With the above set up,
$X$ is a {\it Euclidean $t$-design of $\mathbb{R}^m$}
if
\begin{equation}
\label{eq:DefEuc}
\sum_{i = 1}^p \frac{\sum_{x \in X_i} w(x)}{|S_i^{m-1}|}
\int_{S_i^{m-1}} f(x)
\rho_i({\rm d}x)
= \sum_{x \in X} w(x) f(x)
\end{equation} 
for every polynomial $f \in \mathcal{P}_t(S)$. 
\end{definition}
As readily seen by the definition,
Euclidean designs can be viewed as
cubature formulae on multiple concentric spheres.

The following is a variation of the Sobolev theorem
for Euclidean designs, which generalizes
the familiar theorem of Neumaier and Seidel~\cite{NS88}.
\begin{theorem}
[\cite{NS11}]
Let $G$ be a subgroup of $O(m)$.
Let $X=\cup_{k=1}^M r_k x_k^G$,
where $x_k\in S^{m-1}$ and $r_k>0$.
Then the following are equivalent: \\
(i) $X$ is a $G$-invariant Euclidean $t$-design of $\mathbb{R}^m$. \\
(ii) $\sum_{x \in X} w(x) ||x||^{2 j} \varphi(x) = 0$
for any $\varphi \in {\rm Harm}_l (\mathbb{R}^m)^G$,
$1 \leq l \leq t, 0 \leq j \leq \lfloor \frac{t-l}{2} \rfloor$. 
\label{thm:NS11}
\end{theorem}

Hereafter let $G$ be
an irreducible reflection group in $\mathbb{R}^{m}$.
Such groups are completely classified~\cite{B02}.
Let integers
$1= d_1 \leq d_2 \leq \cdots \leq d_m$
be the exponents of $G$
(see \cite[Ch.V, $\S 6$ ]{B02}). 
\begin{theorem}
[\cite{GS82}] \label{dim_inv}
Let $G$ be a finite irreducible reflection group.
Let $q_i = \dim ({\rm Harm}_i(\mathbb{R}^m)^G)$.
Then
\[
\sum_{i=0}^{\infty}
q_i \lambda^i = \prod^{m}_{i=2} 
\frac{1}{1-\lambda^{1+d_i}}.
\]
In particular,
for any $x \in \mathbb{R}^m$, the orbit
$x^G$ is a spherical $d_2$-design in $S^{m-1}$.
\end{theorem}

Let $\alpha_1, \ldots, \alpha_m$ be
the fundamental roots of a reflection group $G$. 
The {\it corner vectors} $v_1, \ldots, v_m$ are defined by $v_i
\perp \alpha_j$ if and only if $i \ne j$. 
We may assume that $\| v_k \|_2 = 1$. 
We consider the set  

\[ \mathcal{X}(G,J)=\bigcup_{k \in J} r_k v_k^G,
 \]
where
$J\subset \{1,2,\ldots,m\}$ and $r_k>0$. 
Let $R$ denote the set of $r_k$. 

\begin{theorem}
\label{thm:Bajnok}
{\rm (Bajnok~\cite[Theorem 3]{B07})}.
Let $m \ge 2$ be an integer.
Then there is no choice of $R, J$, and $w$ for which
$(\mathcal{X}(B_m,J), w)$ is a Euclidean $8$-design of $\mathbb{R}^m$.
\end{theorem}

Similar results are known
for the groups $A_{m-1}, D_m$~\cite{NS11}.
In Section 5,
we generalize these results, and determine
the maximum degree of invariant Euclidean designs
for all irreducible reflection groups.

\subsection{The Victoir method}

\subsubsection{Combinatorial tools}

Let $K$ be a set of positive integers $k_1, \ldots, k_\ell$.
A pair of $v$ elements $V$, and subsets $\mathcal{B}$ of $V$ of
cardinalities from $K$
is called a {\it $t$-wise balanced design},
denoted by $t$-$(v, K, \lambda)$, if
every $t$ elements of $V$ occur exactly $\lambda$ times in $\mathcal{B}$.
Elements of $V$ and $\mathcal{B}$ are called {\it points} and {\it blocks}.
In particular if $K$ is a singleton, say $K = \{k\}$,
a $t$-wise balanced design is called a {\it $t$-design}, and
is denoted by $t$-$(v, k, \lambda)$.
In this paper we only consider designs without repeated blocks.

It is well known (cf.~\cite{GKL07}) that
for $0 \le t' \le t$ and a subset $T' \subset V$ of $t'$ elements,
the number of blocks of a $t$-$(v, k, \lambda)$ design
containing
$T'$
is given as
\begin{equation}
\label{eq:Tdes1}
\lambda \frac{\binom{v-t'}{t-t'}}{\binom{k-t'}{t-t'}}
= \frac{(v-t') (v-t'-1) \cdots (v-t+1)}{(k-t') (k-t'-1) \cdots (k-t+1)},
\end{equation}
not depending on the choice of $T'$.
For each $0 \le t' \le t$,
a $t$-design is also a $t'$-design.
In general $t$-wise balanced designs do not necessarily have this property;
see Section 3 for the detail.

Let $(V, \mathcal{B})$ be a $t$-wise balanced design with $v$ points and $b$ blocks.
An incidence matrix $M$ of the design $(V, \mathcal{B})$ is
a $v \times b$ zero-one matrix which has a row for each point and
a column for each block, and
for $x \in V$ and $B \in \mathcal{B}$, $(x, B)$-entry takes $1$ iff $x \in B$.
Given real numbers $\alpha, \beta$,
let ${\bf v}_l(\alpha, \beta)$ be a $v$-dimensional vector such that
the first $l$ coordinates are $\alpha$ and the remaining $v-l$ coordinates are $\beta$.
For example,
${\bf v}_l(\alpha, 0)^{B_v}$ means the vertices of
a generalized hyperoctahedron that
is inscribed in the $(v-1)$-dimensional sphere of radius $\sqrt{l \alpha^2}$~\cite{B07}.
To the matrix
$M$,
we associate a {\it generalized incidence matrix with parameters $\alpha, \beta$} by
defining $I_{\alpha, \beta} = \beta J_{v, b} + (\alpha - \beta) M$, where
$\alpha \ne \beta$ and $J_{v, b}$ is the all-one matrix of size $v \times b$.

An $N \times l$ matrix with entries $\pm 1$
is called
an {\it orthogonal array with strength $t$, constraints $l$ and index $\lambda$}, if
in every $t$ columns,
each of the $2^t$ ordered combinations of elements $\pm1$ appears in exactly
$\lambda$ rows.
We denote this by $OA(N, l, 2, t)$.
We do not put $\lambda$ in the notation, since
$\lambda = N/2^t$ by the definition.
When $l \le t$, we allow {\it trivial OA}, namely,
the $2^l \times l$ matrix such that
every $2^l$ ordered combinations of elements $\pm1$
appears in exactly one row.

\subsubsection{Victoir's method}

The group $B_m$ contains two special subgroups:
the subgroup $L$ of all transpositions of coordinates in $\mathbb{R}^m$,
the subgroup $\acute L$ of all sign changes which
is isomorphic to the elementary abelian $2$-group $(\mathbb{Z}/2\mathbb{Z})^m$.
It turns out that $|y^{\acute L}| = 2^{|{\rm wt}(y)|}$, where
${\rm wt}(y)$ is the number of nonzero coordinates of a vector $y$.

We denote by $\mathcal{I}$ the Gaussian integral
$$
\mathcal{I}[f] = \tfrac{1}{(2\pi)^{m/2}} \int_{\mathbb{R}^m}
f((x_1^2, \cdots, x_m^2))
\exp ( - \tfrac{ \| x \|_2^2}{2} ) {\rm d} x_1 \cdots {\rm d} x_m.
$$
This is equivalent to the 
integral $\acute{\mathcal{I}}$ on the first orthant $\mathbb{R}_+^m$
$$
\acute{\mathcal{I}}[f] = \tfrac{1}{(2\pi)^{m/2}} \int_{\mathbb{R}_+^m}
f((x_1, \cdots, x_m))
\exp ( - \tfrac{\| x \|_1}{2} )
( \prod_{i=1}^m x_i )^{-1/2}
{\rm d} x_1 \cdots {\rm d} x_m.
$$
Let
$$
x^2 = (x_1^2, \cdots, x_m^2)
$$
for $x = (x_1, \ldots, x_m) \in \mathbb{R}^m$, and
$$
\sqrt{x} = (\sqrt{x_1}, \cdots, \sqrt{x_m})
$$
for $x = (x_1, \ldots, x_m) \in \mathbb{R}_+^m$.
\begin{proposition}
{\rm (cf.~\cite{V04,X00-3})}.
\label{prop:Victoir2}
If $z_1^{\acute L}, \ldots, z_e^{\acute L}$ and $w_1, \ldots, w_e$ generate
an ${\acute L}$-invariant cubature formula of degree $q$ for $\mathcal{I}$, then
$z_1^2, \ldots, z_e^2$ and
$w_1 2^{{\rm wt}(z_1)}, \ldots, w_e 2^{{\rm wt}(z_e)}$
form a cubature formula of degree $q/2$ for $\acute{\mathcal{I}}$.
Conversely,
if $z_1, \ldots, z_e$ and $w_1, \ldots, w_e$
form a cubature of degree $q/2$
for $\acute{\mathcal{I}}$, then
$\sqrt{z_1}^{\acute L}, \ldots, \sqrt{z_e}^{\acute L}$ and
$w_1/2^{{\rm wt}(z_1)}, \ldots, w_e/2^{{\rm wt}(z_e)}$ generate
a cubature of degree $q$ for $\mathcal{I}$.
\end{proposition}

The following theorem is due to Victoir~\cite[Subsection 4.4]{V04}.
\begin{theorem}
\label{thm:VictoirMethod}
(i) Assume that there exist a cubature formula of degree $q/2$
for $\acute{\mathcal{I}}$ of the form
\begin{align*}
\acute{\mathcal{I}}[f]
=
\frac{w}{\binom{m}{k}}
     \sum_{x \in {\bf v}_k(\alpha, \beta)^L} f(x)  
  + \sum_{i=1}^M \frac{w_i}{| x_i^L|} \sum_{x \in x_i^L} f(x),
\end{align*}
and a
$q/2$-design
with $m$ points and
$b$ blocks of size $k$.
Let $X$ be the columns of a generalized incidence matrix with
parameters $\alpha, \beta$.
Then,
\begin{align*}
\acute{\mathcal{I}}[f]
 = \frac{w}{b} \sum_{x \in X} f(x) +
\sum_{i=1}^M \frac{w_i}{|x_i^L|} \sum_{x \in x_i^L} f(x)
\end{align*}
is a cubature formula of degree $q/2$.
\\
(ii)
Assume that there exist
an ${\acute L}$-invariant cubature formula of degree $q$ for $\mathcal{I}$
of the form
\begin{equation*}
\mathcal{I}[f]
= \sum_{i=1}^M \frac{\lambda_i}{2^{{\rm wt}(x_i)} }
   \sum_{x \in  x_i^{\acute L} } f(x),
\end{equation*}
and $OA(|X_i|, {\rm wt}(x_i), 2, q)$
with rows $X_i$ for $i=1, \ldots, M$. Then,
\begin{equation*}
\mathcal{I}[f]
= \sum_{i=1}^M \frac{\lambda_i}{|X_i|} \sum_{x \in X_i} f(x)
\end{equation*}
is a cubature formula of degree $q$.
\end{theorem}

The Victoir method was originally written in a more general setting.
For example, the integrals considered there are
not restricted to Gaussian integrals.
In this paper, however, we took only Gaussian integrals since
Victoir's ideas can be fully understood with Gaussian integrals.

\section{Generalizing the Victoir method}

In this section we generalize the Victoir method
with a strengthening of the concept of $t$-wise balanced designs.
We use the notations
$B_m, L, {\acute L}$, $\mathcal{I}, \acute{\mathcal{I}},
{\bf v}_i(\cdot, \cdot), {\rm wt}(\cdot)$
that are defined in Subsection 2.4.

A $t$-wise balanced design $(V, \mathcal{B})$ is said to be {\it regular} if
for each $0 \le t' \le t$ and each $t'$-subset $T'$ of $V$,
the number of blocks
containing $T'$
does not depend on the choice of
$T'$~\cite{FKJ89}.
As noted in Subsection 2.4,
any $t$-design possesses this property, but
$t$-wise balanced designs do not always so.
When $t = 2$, this concept is equivalent to
that of {\it equireplicate} $2$-wise balanced designs~\cite{GJ83}.

Let $\mathcal{B}$ be the set of blocks
of a regular $t$-$(v, K, \lambda)$ design,
where $K = \{k_1, \ldots, k_f\}$.
Let $\mathcal{B}_i = \{B \in \mathcal{B} \mid |B| = k_i\}$.
Let $y_i \in \mathbb{R}^m$ with
${\rm wt} (y_i) = k_i$, and $y_K = \{y_1, \ldots, y_f\}$.
We define the following discrete measure:
$$
\displaystyle \delta_{y_K, L}
:=
\sum_{i=1}^f
\frac{|\mathcal{B}_i|}{|\mathcal{B}| \binom{m}{k_i}}
\sum_{x \in y_i^L} \delta_x.$$
\begin{proposition}
\label{prop:Sawa1}
Assume that there exists a regular
$t$-$(m, \{k_i \mid 1 \le i \le f\}, \lambda)$ design $(V, \mathcal{B})$.
Let $X$ be the columns
of a generalized incidence matrix with parameters
$\alpha, \beta$ with $\alpha \ne \beta$.
Let $y_1, \ldots, y_f \in X$ such that ${\rm wt} (y_i) = k_i$.
Then
$$
\int_{ \bigcup_{i=1}^f y_i^L } f(x)
\delta_{y_K, L} ({\rm d} x)
=
\frac{1}{|\mathcal{B}|} \sum_{x \in X} f(x)
$$
for every $f \in \mathcal{P}_t(\bigcup_{i=1}^f y_i^L)$.
\end{proposition}
\begin{proof}
By changing variables $x_i \rightarrow (x_i - \beta)/(\alpha - \beta)$,
there is no loss of generality in assuming $\alpha = 1, \beta = 0$.
Then for any $e_1, \ldots, e_m \ge 0$,
$$
\int_{ \bigcup_{i=1}^f y_i^L } f(x_1^{e_1}, \cdots, x_m^{e_m})
\delta_{y_K, L} ({\rm d} x)
=
\int_{ \bigcup_{i=1}^f y_i^L } f(x_1, \cdots, x_m)
\delta_{y_K, L} ({\rm d} x).
$$
Permuting the rows of an incidence matrix
also gives another $t$-wise balanced design with the same parameters $m, k_1, \ldots, k_f, \lambda$.
Thus it suffices to show that
$$
\int_{ \bigcup_{i=1}^f y_i^L } f(x)
\delta_{y_K, L} ({\rm d} x)
=
\frac{1}{|\mathcal{B}|} \sum_{x \in X} f(x)
$$
for the monomials
$f(x) = \prod_{i=1}^j x_i, 1 \le j \le t$.
To do this,
we count the pairs
$(T', B) \in \binom{V}{t'} \times \mathcal{B},\; T' \subset B$
in two ways:
\begin{align*}
\lambda' \binom{m}{t'}
= \sum_{T' \in \binom{V}{t'}} \sum_{T' \subset B \in \mathcal{B}} 1
= \sum_{B \in \mathcal{B}} \sum_{T' \subset B \atop{T' \in \binom{V}{t'}}} 1
= \sum_{i=1}^f
\sum_{B \in \mathcal{B}_i}
\sum_{T' \subset B \atop{T' \in \binom{V}{t'}}} 1
=
\sum_{i=1}^f
|\mathcal{B}_i|
\binom{k_i}{t'},
\end{align*}
where the regularity is used to show the first equality.
Thus, for $f(x) = \prod_{i=1}^{t'} x_i$,
\begin{align*}
\displaystyle
\sum_{y \in X} f(y)
=
\lambda'
=
\sum_{i=1}^f
\frac{|\mathcal{B}_i| \binom{k_i}{t'}}{\binom{m}{t'}}.
\end{align*}
This is further transformed to
\begin{multline*}
\sum_{i=1}^f
\frac{|\mathcal{B}_i|}{\binom{m}{k_i}}
\cdot
\frac{\binom{m}{k_i} \binom{k_i}{t'}}{\binom{m}{t'}}
=
\displaystyle
\sum_{i=1}^f
\frac{|\mathcal{B}_i|}{\binom{m}{k_i}}
\cdot
\binom{m - t'}{k_i - t'} \\
=
\sum_{i=1}^f
\frac{|\mathcal{B}_i|}{\binom{m}{k_i}}
\sum_{x \in y_i^L} f(x)
=
|\mathcal{B}| \cdot
\int_{ \bigcup_{i=1}^f y_i^L } f(x)
\delta_{y_K, L} ({\rm d} x).
\end{multline*}
\end{proof}

\begin{remark}
\label{rem:Kageyama}
In a combinatorial framework (cf.~\cite{S90}),
some researchers regard $t$-wise balanced designs
as cubature on ^^ ^^ discrete spheres".
However, among them,
there are only a few publications where
the regularity of designs is mentioned.
Victoir seems to be the first who employed combinatorial $t$-designs
to reduce the size of cubature for ordinary continuous integrals.
\end{remark}

The following generalizes
Theorem~\ref{thm:VictoirMethod}~(i)
and motivates the study of regular $t$-wise balanced designs
both in a combinatorial and analytic manner.

\begin{theorem}
\label{thm:Sawa1}
Assume that
there exists a regular $q/2$-wise balanced design with
$m$ points and $b_i$ blocks of size $k_i$, $i=1, \ldots, e$.
Moreover assume that there exists
a cubature formula of degree $q/2$ (or index $q/2$)
for $\acute{\mathcal{I}}$ of the form
\begin{align*}
\acute{\mathcal{I}}[f]
=  c \left( \sum_{i=1}^e
\frac{b_i}{\binom{m}{k_i} b}
     \sum_{x \in
{\bf v}_{k_i}(\alpha, \beta)^L} f(x) \right)
  + \sum_{i=2}^M \frac{w_i}{| x_i^L|} \sum_{x \in x_i^L} f(x)
\end{align*}
where $b$ is the total number of blocks of the design and $c$ is a positive number.
Let $X$ be the columns of a generalized incidence matrix with
parameters $\alpha, \beta$.
Then
\begin{align*}
\acute{\mathcal{I}}[f]
 =  \frac{c}{b} \sum_{x \in X} f(x) +
\sum_{i=2}^M \frac{w_i}{|x_i^L|} \sum_{x \in x_i^L} f(x)
\end{align*}
is a cubature formula of degree $q/2$ (or index $q/2$).
\end{theorem}

The following proposition is often used in Section 4.
\begin{proposition}
\label{prop:SX1}
Assume there exists
a $t$-$(v, k, \lambda)$ design.
Then the following hold: \\
(i) There exists
a regular $t$-$(v-1, \{k, k-1\}, \lambda)$ design with
$\lambda \tfrac{\binom{v-1}{t-1}}{\binom{k-1}{t-1}}$
blocks of size $k-1$ and
$\tfrac{(v-k) \lambda}{k} \tfrac{\binom{v-1}{t-1}}{\binom{k-1}{t-1}}$
blocks of size $k$. \\
(ii) 
Let $X$ be the columns of an incidence matrix of the design given in (i), and
$y_1 = {\bf v}_k(1, 0),\;y_2 = {\bf v}_{k-1}(1, 0)$.
Then for every $f \in \mathcal{P}_t(y_1^L \cup y_2^L)$,
\begin{equation*}
\sum_{x \in y_1^L \cup y_2^L} f(x)
=
\tfrac{\binom{k-1}{t-1} \binom{v-1}{k-1}}{\lambda \binom{v-1}{t-1}}
\sum_{x \in X} f(x).
\end{equation*}
\end{proposition}
\begin{proof}
(i) Let $(V, \mathcal{B})$ be a $t$-$(v, k, \lambda)$ design, and
$x \in V$. We consider
the incidence structure $(V', \mathcal{B}')$, where
$$
V' = V\setminus\{x\},\quad
\mathcal{B}'
= \{B \in \mathcal{B} \mid x \notin B\} \cup \{B\setminus\{x\}
\mid x \in B \in \mathcal{B}\}.
$$
Then
$(V', \mathcal{B}')$ is a regular $t$-wise balanced design
with parameters determined by (\ref{eq:Tdes1}).
(ii) The assertion follows by (i) and Proposition~\ref{prop:Sawa1}.
\end{proof}

We close this section with some remarks on
regular $t$-wise balanced designs.
First, as far as the authors know,
there are only a few general results on
the existence of regular $t$-wise balanced designs
for $t \ge 3$.
Some examples are known, most of which are obtained
by trivial ways as Proposition~\ref{prop:SX1}~(i).
The second author and Reinhard Laue
searched for regular $3$-, $4$- and $5$-wise balanced designs
with Discreta, a sophisticated program to compute designs,
and found many designs with small parameters,
some of which are summarized in Table~\ref{tbl:des}.
\begin{table}[h]
\caption{Some new regular $t$-wise balanced designs}
\label{tbl:des}
\begin{center}
\begin{tabular}{|c|c|}
\hline
Parameters & Groups \\
\hline
$3$-$(25, \{6,10\}, 4)$ & $AGL(1, 25)$ \\
\hline
$4$-$(27, \{5,8\}, 5)$ & $ASL(3,3)$ \\
\hline
$5$-$(33, \{6,7\}, 10)$ & $P\Gamma L (2, 32)$ \\
$5$-$(33, \{6,8\}, 20)$ & $P\Gamma L (2, 32)$ \\
$5$-$(33, \{6,9\}, 15)$ & $P\Gamma L (2, 32)$ \\
$5$-$(33, \{7,10\}, 42)$ & $P\Gamma L (2, 32)$ \\
\hline
$5$-$(55, \{6,5\}, 5)$ & $C_2 \times P\Gamma L (2, 27)$ \\
\hline
\end{tabular}
\end{center}
\end{table}
We believe that
there will be further nontrivial regular $t$-wise balanced designs.
However, in this paper,
such thorough discussions are omitted and
left for future work.

A natural problem is to find a good bound for the number of blocks
of a $t$-wise balanced design.
Ziqing Xiang, a student of Eiichi Bannai, recently derived the Fisher-type bound
for regular $t$-wise balanced designs.
Namely, he showed that
if there is a regular $2e$-wise balanced design $(V, \mathcal{B})$
with $f$ distinct sizes of blocks, then
$$
|\mathcal{B}| \ge \sum_{i=0}^{f-1} \binom{|V|}{e-i}.
$$
This bound is sharp when
$t=2$ and $f=2$, by a result of Woodal~\cite{W70}.
Moreover, when $t=4$ and $f = 2$,
a tight example can be constructed
from the ordinary tight $4$-design which corresponds to the Johnson scheme.
Without regularity,
no good bounds seem to be known
\footnote{Eiichi Bannai kindly told us detailed informations
on bounds for regular $t$-wise balanced designs
through email conversation.}.

\section{Cubature arising from Victoir's method and its generalization}

In this section many cubature formulas are constructed
by Victoir's method and its generalization formulated in Section 3.

\subsection{Index-four cubature}

There are many publications on the existence of
index-four cubature in small dimensional spaces that
are not minimal but have few points;
see, e.g.,~\cite{R95},~\cite{S71}.
In general dimensional cases, however,
it seems that
explicit constructions of good cubature are not enough known
\footnote{Oksana Shatalov and Yuan Xu kindly told us these informations.}.
Therefore the following theorem by Shatalov~\cite{S01} is very important.
\begin{theorem}
\label{thm:Shatalov}
(i)
{\rm (\cite[Theorem 4.4.9]{S01})}. 
Assume that for given $m, n$, and $q$,
there exists a cubature of index $q$ with $n$ points on $S^{m-1}$.
Then for any $M \ge m$,
there exists a cubature of index $q$ with
$((q+2)/2)^{M-m} n$ points on $S^{M-1}$. \\
(ii)
{\rm (\cite[Corollary 4.4.12]{S01})}. 
There exists an index-four cubature on $S^{m-1}$ with $n$ points when
\begin{align}
m = 2^{2l} + s, \quad n = 2^{2l} \cdot 3^s \cdot (2^{2l-1} + 1), & \qquad l \ge 1,\; s \ge 0. \label{eq:Shatalov1} \\
m = 2l+2+s, \quad n = 3^{s+1} \cdot ((l+1)^2 + 1), & \qquad \text{$l$ is a prime power},\;s \ge 0. \label{eq:Shatalov2}
\end{align}
\end{theorem}

\begin{remark}
For $s = 0$, Theorem~\ref{thm:Shatalov} is
a theorem of K\"onig~\cite{K95}.
Family (\ref{eq:Shatalov1}) improves
the upper-bound part of (\ref{eq:bound})
if $s$ is fixed and $m$ is sufficiently large,
or $s = 1, 2$.
A similar conclusion holds for (\ref{eq:Shatalov2}).
\end{remark}

Cubature formulae in general-dimensional spaces
that improve Shatalov's families
are constructed.

\begin{theorem}
\label{thm:Main2}
(i)
Let $l \ge 2, m$ be integers.
Assume that
$$
\ell =
\left\{\begin{array}{cc}
4l-1 & \qquad \text{if \quad $2^{2l-1} \le m \le 2^{2l}$}; \\
4l+1 & \qquad \text{if \quad $2^{2l} < m < 2^{2l+1}$}. \\
\end{array} \right.$$
Then there is an integer $n$ with $2^{\ell-1} + m < n \le 2^\ell + m$
for which an index-four cubature with $n$ points
on $S^{m-1}$ exists. \\
(ii)
Let $l \ge 2, l', m$ be integers.
Assume that
$m \in \{3^{l'+2} - 2, 2\cdot9^{l'+1} - 2\}$, and
$$
\ell =
\left\{\begin{array}{cl}
4l-1 & \quad \text{if $2^{2l-1} \le (m+2)/3 \le 2^{2l}$}; \\
4l+1 & \quad \text{if $2^{2l} < (m+2)/3 < 2^{2l+1}$}. \\
\end{array} \right.
$$
Then there is an integer $n$ with $2^{\ell-1} m < n \le 2^\ell m$
for which an index-four cubature with $n$ points on $S^{m-1}$ exists.
\end{theorem}

The following lemma is employed, where
the proof is easy and so omitted.

\begin{lemma}
\label{lem:SXind41}
The following is an $m$-dimensional index-two cubature
for $\acute{\mathcal{I}}$. \\
(i) 
For $m \ge 3$,
\begin{equation*}
\acute{\mathcal{I}}[f] =
\tfrac{1}{2m} \sum_{x \in {\bf v}_1(\sqrt{4m}, 0)^L} f(x) +
\tfrac{1}{2} \sum_{x \in {\bf v}_m(\sqrt{2}, 0)^L} f(x).
\end{equation*}
(ii) 
For $m \equiv 1 \pmod 3$,
\begin{equation*}
\acute{\mathcal{I}}[f] =
\tfrac{1}{\binom{m}{(m+2)/3}} \sum_{x \in {\bf v}_{(m+2)/3}(\sqrt{\tfrac{9m}{m+2}}, 0)^L} f(x).
\end{equation*}
\end{lemma}

More $B_m$-invariant cubature can be obtained systematically by
using the Sobolev theorem.

\noindent
{\it Proof of Theorem~\ref{thm:Main2}~(i)}.
Take an $OA(2^{4l}, m, 2, 4)$ if $2^{2l-1} \le m \le 2^{2l}$,
and an $OA(2^{4l+2}, m, 2, 4)$ if $2^{2l} < m < 2^{2l+1}$.
These OA are constructed from
an $OA(2^{4l}, 2^{2l}, 2, 4)$ and an $OA(2^{4l+2}, 2^{2l+1}-1, 2, 4)$ which
are the dual of
the Kerdock code and the BCH code over $\mathbb{F}_2$
(cf.~\cite[p.~102, p.~94]{HSS99}) respectively,
where $0, 1 \in \mathbb{F}_2$ are replaced by $-1, 1$.
Hence,
by Theorem~\ref{thm:VictoirMethod}~(ii),
Lemma~\ref{lem:SXind41}~(i) and Proposition~\ref{prop:Victoir2},
we get an index-four cubature for $\mathcal{I}$ with
at most $2^{\ell+1} + 2m$ points.
The Kerdock OA has central symmetry (cf.~\cite{K95}).
The BCH OA is also centrally symmetric since it is linear.
The result follows by Propositions~\ref{prop:SSI} and~\ref{prop:SX1}.

\noindent
{\it (ii)}
The existence of a $2$-$(m, (m+2)/3, (m+2)/9)$ design
with $m$ blocks is known~\cite{IT07}.
So, by Theorem~\ref{thm:VictoirMethod}~(i) and Lemma~\ref{lem:SXind41} (ii),
we obtain an index-two cubature for $\acute{\mathcal{I}}$ with $m$ points.
According to Proposition~\ref{prop:Victoir2},
the resulting cubature is equivalent to
an ${\acute L}$-invariant cubature of index $4$ with $2^{(m+2)/3} m$.
Applying Theorem~\ref{thm:VictoirMethod}~(ii) to
this formula and the OA given in the proof of Theorem~\ref{thm:Main2},
we have an index-four cubature for $\mathcal{I}$
with at most $2^{\ell+1} m$ points.
Since
the Kerdock and BCH OA have central symmetry,
the result follows by Propositions~\ref{prop:SX1} and~\ref{prop:SSI}.
$\Box$

More general-dimensional index-four cubature
with $O(m^2)$ or $O(m^3)$ points can be obtained
by using suitable OA, $2$-designs, and
regular pairwise balanced designs.

\begin{remark}
\label{rem:SawaXuSeries}
(i)
Theorem~\ref{thm:Main2} improves Theorem~\ref{thm:Shatalov}
for many values of $m$.
When $2^{2l-1} \le m \le 2^{2l}$,
the family of Theorem~\ref{thm:Main2}~(i) comes from
centrally symmetric cubature by ^^ ^^ halving" opposite row-vectors of OA.
The underlying symmetric cubature
were found by Victoir~\cite[Subsection 5.3]{V04}.
(ii) 
Theorem~\ref{thm:Main2} does not mention
the exact number of points of the constructed cubature.
When $m = 2^{2l}$ in Theorem~\ref{thm:Main2}~(i),
the underlying OA is the Kerdock OA and
no two distinct rows coincide. So,
the constructed cubature has exactly $2^{4l-1} + 2^{2l}$ points, which
is equivalent to K\"onig's family.
(iii) 
By Proposition~\ref{prop:Victoir2}
the $L$-invariant formula of Lemma~\ref{lem:SXind41}~(i)
is equivalent to the degree-five cubature of Stroud~\cite{S71}.
Moreover
the formula (ii) corresponds to K\"ursch\'ak's identity
in number theory; see Section 6.
\end{remark}

\subsection{Index-six cubature}

Shatalov~\cite[Theorem~4.7.20]{S01} compiled
known index-six cubature with few points in small-dimensional spheres
as Table~\ref{tbl:1} (strictly speaking, a part of the original).
\begin{table}[h]
\caption{Index-six cubature on $S^{m-1}$ with $n$ points}
\label{tbl:1}
\begin{center}
\begin{tabular}{|c|c|c|c|c|c|c|c|c|c|c|c|c|c|c|}
\hline
No  &  1  &  2  &  3  &  4  &  5  &  6  &  7  &  8   &   9  &  10  &  11  &  12   & 13     \\
\hline
$m$ & $3$ & $4$ & $5$ & $6$ & $7$ & $8$ & $9$ & $10$ & $11$ & $16$ & $17$ & $18$  & $23$   \\
\hline
  $n$ & $11$&$23$ & $41$&$63$ &$113$&$120$&$480$&$1920$&$7680$&$2160$&$8640$&$34650$&$2300$ \\
\hline
\end{tabular}
\end{center}
\end{table}
Nos.~1,~2,~4 are respectively in~\cite{R92},~\cite{HS94},~\cite{GS82}
\footnote{The existence of $23$-point
cubature of index $6$ on $S^3$ is not covered
in~\cite[Theorem~4.7.20]{S01}.}.
Nos.~3,~5 are in~\cite{S71}, and No.~6 in~\cite{DGS77}.
To complete Table~\ref{tbl:1},
Shatalov applied Theorem~\ref{thm:Shatalov}~(i) to one of the above formulae.
For example,
No.~7 has $4$ times as many points as No.~6 does.
According to Shatalov,
Table~\ref{tbl:1} had not been updated so far,
and the existence of general-dimensional index-six cubature
with few points is not fully known

Two families of general-dimensional cubature
that improve the upper-bound part of (\ref{eq:bound})
are given.

\begin{theorem}
\label{thm:Main4}
Let $Q$ be a prime power such that
$Q \equiv 1 \pmod{6}, Q \ne 25$.
Let $m \in \{Q+1, Q\}$ and $l$ be an integer with
$l \ge 3, 2^{2l-2} < m \le 2^{2l}$.
Then there is an integer $n \le 2^{6l-2} (3Q(Q+1) + 1) + m$
for which
an index-six cubature with $n$ points on $S^{m-1}$ exists.
\end{theorem}

\begin{lemma}
\label{lem:SXind62}
The following is an $m$-dimensional index-three cubature for $\acute{\mathcal{I}}$.
(i) 
For $m \equiv 2 \pmod 6$ and $8 \le m$,
\begin{align*}
\acute{\mathcal{I}}[f] &=
 \tfrac{1}{3} \sum_{x \in {\bf v}_{m}(\sqrt[3]{\tfrac{12}{5}}, 0)^L} f(x)
+ \tfrac{1}{3m} \sum_{x \in {\bf v}_1(\sqrt[3]{\tfrac{216m}{m+4}}, 0)^L} f(x)
\\ & \qquad
+ \tfrac{1}{3 \binom{m}{(m+10)/6}}
\sum_{x \in {\bf v}_{(m+10)/6}(\sqrt[3]{\tfrac{1296m(m-1)}{(m+4)(m+10)}}, 0)^L} f(x)
\end{align*}
(ii)
For $m \equiv 1 \pmod 6$ and $7 \le m$,
\begin{align*}
\acute{\mathcal{I}}[f]
&=
\tfrac{1}{3} \sum_{x \in {\bf v}_m(\sqrt[3]{\tfrac{9}{5}}, 0)^L} f(x)
+ \tfrac{1}{3m} \sum_{x \in {\bf v}_1(\tfrac{1}{3m}, 0)^L} f(x) \\
&
+ \tfrac{1}{ 3 \binom{m+1}{(m+11)/6}}
\sum_{x \in
{\bf v}_{(m+11)/6}(\sqrt[3]{\tfrac{1296m(m+1)}{(m+5)(m+11)}}, 0)^L \bigcup
{\bf v}_{(m+5)/6}(\sqrt[3]{\tfrac{1296m(m+1)}{(m+5)(m+11)}}, 0)^L} f(x).
\end{align*}
\end{lemma}

\noindent
{\it Proof of Theorem~\ref{thm:Main4}}. 
First we consider the case where $m = Q+1$.
There exists a $3$-$(Q+1, (Q+11)/6, (Q+5)(Q+11)/72)$ (cf.~\cite{GKL07}), which
has $3Q(Q+1)$ blocks by (\ref{eq:Tdes1}).
By Theorem~\ref{thm:VictoirMethod}~(i) and Lemma~\ref{lem:SXind62}~(i),
we obtain an index-three cubature for $\acute{\mathcal{I}}$ with
$1 + (Q+1) + 3Q(Q+1)$ points.
By Proposition~\ref{prop:Victoir2},
this is equivalent to an ${\acute L}$-invariant cubature with
$2^{Q+1} + 2(Q+1) + 2^{(Q+11)/6} \cdot 3Q(Q+1)$ points.
By applying Theorem~\ref{thm:VictoirMethod}~(ii) to
an $OA(2^{6l-1}, Q+1, 2, 7)$ and an $OA(2^{6l-1}, (Q+11)/6, 2, 7)$
that are subarrays of the dual
$OA(2^{6l-1}, 2^{2l}, 2, 7)$ of
the Delsarte-Goethals code (cf.~\cite[p.~103]{HSS99}),
we obtain an index-six formula for $\mathcal{I}$ with
at most $2^{6l-1} \cdot (1 + 3Q(Q+1)) + 2(Q+1)$ points.
Note that
the $OA(2^{6l-1}, 2^{2l}, 2, 7)$ has central symmetry.
In fact, the Delsarte-Goethals code can be constructed by
applying the Gray-code mapping
$0 \mapsto 00, 1 \mapsto 01, 2 \mapsto 11, 3 \mapsto 10$
to linear, cyclic codes over $\mathbb{Z}_4$.
Replacing $0, 1 \in \mathbb{F}_2$ by $\pm1$ implies the
central symmetry of the OA.
The result thus follows by Propositions~\ref{prop:SX1} and~\ref{prop:SSI}.
Similar arguments work when $m = Q$;
replace the above $L$-invariant formula by that of Lemma~\ref{lem:SXind62}~(ii).
By Proposition~\ref{prop:SX1}~(ii)
the above $3$-design can be reduced to
a regular $3$-wise balanced design with $Q$ points
and $3Q(Q+1)$ blocks.
By Theorem~\ref{thm:Sawa1}
we obtain an index-three cubature for $\acute{\mathcal{I}}$
with $1 + 3Q(Q+1) + Q$ points.
Then the assertion follows by the same argument as in the case $m = Q+1$.
$\Box$

\begin{remark}
The family of Theorem~\ref{thm:Main4}
has $O(m^5)$ points, improving
the upper-bound part of (\ref{eq:bound}).
More general-dimensional index-six cubature with $O(m^5)$ points
may be obtained by using known infinite families of $3$-designs
~\cite{GKL07}.
\end{remark}

Two more interesting cubature are given.

\begin{example}
\label{exam:SXind61}
The following is a
$7$-dimensional index-three cubature for $\acute{\mathcal{I}}$:
\begin{align}
\label{eq:SXind62}
\acute{\mathcal{I}}[f]
&=
\tfrac{1}{140} \sum_{x \in {\bf v}_4(\sqrt[3]{28}, 0)^L
       \cup {\bf v}_3(\sqrt[3]{28}, 0)^L} f(x)
+ \tfrac{1}{14} \sum_{x \in {\bf v}_1(\sqrt[3]{112}, 0)^L} f(x).
\end{align}
A $3$-$(8, 4, 1)$ design exists (cf.~\cite{GKL07}), and so does
a regular $3$-$(7, \{4, 3\}, 1)$ design with
$7$ blocks of sizes $4$ and $3$ according to Proposition~\ref{prop:SX1}~(i).
Let $X$ be the columns of an incidence matrix of the $3$-wise balanced design.
By Proposition~\ref{prop:SX1}~(ii),
\begin{equation}
\label{eq:SXind61}
\sum_{x \in {\bf v}_4(\sqrt[3]{28}, 0)^L \cup {\bf v}_3(\sqrt[3]{28}, 0)^L} f(x)
=
5 \sum_{x \in X} f(x)
\end{equation}
for every $f \in \mathcal{P}_3$.
Hence,
by (\ref{eq:SXind62}), (\ref{eq:SXind61}), and
Proposition~\ref{prop:Victoir2},
the following index-six cubature for $\mathcal{I}$ is obtained.
\begin{align}
\label{eq:SX63}
\mathcal{I}[f]
=
\tfrac{1}{448} \sum_{x \in (\sqrt[6]{28} \cdot X_1)^{\acute L}} f(x)
+ \tfrac{1}{224} \sum_{x \in (\sqrt[6]{28} \cdot X_2)^{\acute L}} f(x)
+ \tfrac{1}{28} \sum_{x \in {\bf v}_1(\sqrt[3]{112}, 0)^{B_m}} f(x)
\end{align}
where $X_1 = \{x \in X \mid {\rm wt}(x) = 4\},
X_2 = \{x \in X \mid {\rm wt}(x) = 3\}$.
This is reduced to a $91$-point formula of index $6$ on $S^6$
by Propositions~\ref{prop:SSI} and~\ref{prop:SX1}.
\end{example}

\begin{example}
\label{exam:SXind62}
The following is a $9$-dimensional index-three cubature
for $\acute{\mathcal{I}}$:
\begin{align}
\label{eq:SXind67000}
\acute{\mathcal{I}}[f]
&= \tfrac{1}{3} \sum_{x \in {\bf v}_9(1, 0)^L} f(x)
+ \tfrac{1}{630} \sum_{x \in 
{\bf v}_4(\sqrt[3]{60}, 0)^L \cup {\bf v}_3(\sqrt[3]{60}, 0)^L} f(x) 
+ \tfrac{1}{27} \sum_{x \in {\bf v}_1(\sqrt[3]{180}, 0)^L} f(x). \nonumber
\end{align}
The existence of a $3$-$(10, 4, 1)$ design (cf.~\cite{GKL07})
implies that of a regular $3$-$(9, \{4, 3\}, 1)$ design with
$12$ blocks of size $3$ and $18$ blocks of size $4$.
By the same way as in Example~\ref{exam:SXind61},
a $457$-point formula on $S^8$ is obtained.
\end{example}

\begin{remark}
\label{rem:SXind61}
(i) The formula No.~5 of Table~\ref{tbl:1} implies that
$N(7, 6) \le 113$.
Example~\ref{exam:SXind61} improves this to
\begin{equation}
\label{eq:SXind610}
N(7, 6) \le 91.
\end{equation}
The lower-bound part of (\ref{eq:bound}) shows $84 \le N(7, 6)$.
The authors do not know
the existence of cubature with fewer points than
the $91$-point formula on $S^6$.
It is also noted that
spherical $84$-point index-six cubature on $S^6$
do not exist by Theorem~1 of~\cite{BD80}.
(ii) The formula No.~7 of Table~\ref{tbl:1} implies that
$N(9, 6) \le 480$.
Example~\ref{exam:SXind62} improves this to
\begin{equation}
N(9, 6) \le 457.
\end{equation}
\end{remark}

\bigskip

The fundamental roots of the group $B_m$ are
$\alpha_i = e_i - e_{i+1}$ for $i=1, \ldots, m-1$, and
$\alpha_m = \sqrt{2} e_m$, where
$e_1, \ldots, e_m$ are the standard basis vectors in $\mathbb{R}^m$~\cite{B02}.
The corner vectors are 
$v_i = (1/\sqrt{i}, \cdots, 1/\sqrt{i}, 0, \cdots, 0)$ for $i = 1, \ldots, m$.
We note that
all $B_m$-invariant cubature of indices $4, 6$ given in Section 4
consist of the orbits of the corner vectors.
By Bajnok's theorem,
in order to find higher-index spherical cubature,
we must take at least one orbits of points which
are not corner vectors;
see, e.g.,~\cite{SX11} for a simple construction of higher-index cubature on spheres.

A strengthening of Bajnok's theorem
is proved in the next section.

\section{The maximum strength of invariant Euclidean designs}

We use the notations
$R, J, \alpha_i, v_i$, and $\mathcal{X}{(G,J)}$ that
are defined in Subsection 2.3.
The aim of this section is to prove the following theorem.

\begin{theorem}
\label{thm:NS}
Let $G$ be a finite irreducible reflection group in
$\mathbb{R}^m$ with $m \ge 2$.
Then there is no choice of $R, J$, and
a weight $w$ for which
$(\mathcal{X}(G, J), w)$ is a Euclidean $t$-design of $\mathbb{R}^m$
in the following cases:
\begin{enumerate}
\item[(i)] $t \ge 6$ if $G = A_{m-1}$;
\item[(ii)] $t \ge 8$ if $G = B_m, D_m$;
\item[(iii)] $t \ge 10$ if $G = E_6$;
\item[(iv)] $t \ge 12$ if $G = F_4, H_3, E_7$;
\item[(v)] $t \ge 16$ if $G = E_8$;
\item[(vi)] $t \ge 24$ if $G = H_4$.
\end{enumerate}
\end{theorem}

The following lemma plays an important role to prove the theorem.

\begin{lemma} \label{lem:main_lem}
Let $G$ be a subgroup of $O(m)$, and 
$X=\{x_1,\ldots,x_M\}$
be a subset of $S^{m-1}$. 
Let $\{f_{i,k}\}_{k=1}^{m_i}$ be a basis of 
${\rm Harm}_{2i}(\mathbb{R}^m)^G$, where $m_i=\dim({\rm Harm}_{2i}(\mathbb{R}^m)^G)$. 
Let $V_i$ be the space ${\rm Span}_{\mathbb{R}}\{(f_{i,k}(x_1),\ldots,
f_{i,k}(x_M) ) \mid k = 1, \ldots, m_i\} \subset \mathbb{R}^X$.
Suppose there is 
$v \in \sum_{i=1}^s V_i$ such that
all entries of $v$ are positive. 
Then
there is no choice of radii $r_i$ and a weight $w$ for which 
$(\sum_{i=1}^Mr_i x_i^G,w)$
is a Euclidean $2s$-design.  
\end{lemma}
\begin{proof}
Since $X \subset S^{m-1}$, we can express 
\begin{align*}
v&=\sum_{i=1}^s\sum_{k=1}^{m_i}a_{i,k}(f_{i,k}(x_1),\ldots,
f_{i,k}(x_M)) \\
&=\sum_{i=1}^s\sum_{k=1}^{m_i} a_{i,k}(\|x_1 \|_2^{2s-2i}f_{i,k}(x_1),\ldots,
\|x_n \|_2^{2s-2i}f_{i,k}(x_M)),
\end{align*}
where $a_{i,k}$ are real numbers. Let $f(x):=
\sum_{i=1}^s\sum_{k=1}^{m_i} a_{i,k} \|x \|_2^{2s-2i}f_{i,k}(x)$. 
Then 
$ f \in \sum_{2i+2j=2s, i\geq 1,j \geq 0} \|x \|_2^{2j} 
{\rm Harm}_{2i}(\mathbb{R}^m)^G$, and $f$
satisfies $f(x_i)> 0$ for each $i=1,\ldots, M$.
By noting that 
$f(r_ix_i^g)=r_i^{2s} f(x_i^g) = r_i^{2s} f(x_i) > 0$ for 
$i=1,\ldots, M$, and $g \in G$,
this lemma follows.
\end{proof}

\begin{remark}
If our assumption in Lemma~\ref{lem:main_lem} holds,   
then any subset of $\{rx^g \mid g \in G, x \in X, r>0\}$
does not form
a Euclidean $2s$-design. In particular, for any subgroup $H$ of $G$, 
$(\sum_{i=1}^Mr_i x_i^H,w)$
is not a Euclidean $2s$-design for any radii $r_i$ and weight $w$.
\end{remark}

The proof of Theorem~\ref{thm:NS} is divided into some cases.
The following notations are used.
For a finite irreducible reflection group $G$,  
$v_i$ denotes the corner vector normalized by $(v_i,\alpha_i)=1$, 
$v_i':=v_i/\sqrt{(v_i,v_i)}$, and $N_i:=|v_i^G|$.  
Let $e_i$ be the column vector with the $i$-th entry $1$ and the 
others $0$. Define
\[
{\rm sym}(f):=\frac{1}{|(S_m)_{f}|}\sum_{g \in S_m} f(x^g),
\quad (S_m)_f:=\{g \in S_m \mid f(x^g)=f(x) \}
\]
for an $m$-variable polynomial $f$, where $S_m$ is the symmetric group of $m$ elements. 
Let $p_i := x_2^2+x_3^2+\cdots+x_{i+1}^2$ for $i \ge 2$. 
The polynomials $h_i$ in the following subsections are harmonic.  

\subsection{Group $F_4$}
\quad \\
\textbf{Dynkin diagram}
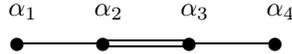
\begin{center}
\unitlength.015in
\begin{picture}(180,25)
\put( 30, 15){\circle*{5}} 
\put( 30, 15){\line(1,0){30}} 
\put( 60, 15){\circle*{5}} 
\put( 60, 16){\line(1,0){30}} 
\put( 60, 14){\line(1,0){30}} 
\put( 90, 15){\circle*{5}} 
\put( 90, 15){\line(1,0){30}} 
\put( 120, 15){\circle*{5}} 

\put( 27, 25){$\alpha_1$}
\put( 57, 25){$\alpha_2$}
\put( 87, 25){$\alpha_3$}
\put( 117, 25){$\alpha_4$}
\end{picture}
\end{center}

\noindent
\textbf{Exponents} 
$\qquad 1,5,7,11.$

\noindent
\textbf{Fundamental roots} 
\begin{align*}
& \alpha_1:={}^te_1-{}^te_2,
\alpha_2:={}^te_2-{}^te_3,
\alpha_3:={}^te_4,
\alpha_4:=\tfrac{-{}^te_1-{}^te_2-{}^te_3+{}^te_4}{2}.
\end{align*}

\noindent
\textbf{Corner Vectors} 
\begin{align*}
& v_1={}^te_1+{}^te_4,
v_2={}^te_1+{}^te_2+2{}^te_4,
v_3={}^te_1+{}^te_2+{}^te_3+3 \ {}^te_4,
v_4=2{}^te_4.
\end{align*}

\noindent
\textbf{Size of Orbit} $\qquad N_1=24,N_2=96,N_3=96,N_4=24.$

\noindent
\textbf{Harmonic Molien series}
\begin{eqnarray*}
\frac{1}{(1-t^6)(1-t^8)(1-t^{12})}=
1+t^6+t^8+2 t^{12}+t^{14}+\cdots. 
\end{eqnarray*}

\noindent
\textbf{$G$-invariant harmonic polynomials}\\
For $i=6,8,12$, ${\rm 
Harm}_{i}(\mathbb{R}^4)^{F_4}$ is spanned by the following: \\
\noindent
1. {\it Degree $6$.} 
\[
f_6:={\rm sym}(x_1^6)-5 {\rm sym}(x_1^4 x_2^2 )+30 {\rm sym}(x_1^2 x_2^2 x_3^2 ).
\]
\\
\noindent
2. {\it Degree $8$.} 
\begin{multline*}
f_8:=
{\rm sym}(x_1^8) -\tfrac{28}{3} {\rm sym}(x_1^6x_2^2)
+\tfrac{98}{3} {\rm sym}(x_1^4x_2^4)
-28 {\rm sym}(x_1^4x_2^2x_3^2) +504x_1^2x_2^2x_3^2x_4^2.
\end{multline*}
\\
\noindent
3. {\it Degree $12$.} 
\begin{multline*}
f_{12,1}:=
{\rm sym}(x_1^{12})-22{\rm sym}(x_1^{10}x_2^2)+79{\rm sym}(x_1^8x_2^4) \\
+258{\rm sym}(x_1^8x_2^2x_3^2) 
-116{\rm sym}(x_1^6x_2^6)-236{\rm sym}(x_1^6x_2^4x_3^2) \\
-4392{\rm sym}(x_1^6x_2^2x_3^2x_4^2)
+570{\rm sym}(x_1^4x_2^4x_3^4)+3660{\rm sym}(x_1^4x_2^4x_3^2x_4^2),
\end{multline*}
\begin{multline*}
f_{12,2}:=
{\rm sym}(x_1^{12})
-22{\rm sym}(x_1^{10}x_2^2)
+\tfrac{133}{2}{\rm sym}(x_1^8x_2^4) \\
+\tfrac{591}{2}{\rm sym}(x_1^8x_2^2x_3^2)
-\tfrac{157}{2}{\rm sym}(x_1^6x_2^6)
-\tfrac{1369}{4}{\rm sym}(x_1^6x_2^4x_3^2) \\
-4167{\rm sym}(x_1^6x_2^2x_3^2x_4^2)
+\tfrac{2265}{2}{\rm sym}(x_1^4x_2^4x_3^4)
+\tfrac{6945}{2}{\rm sym}(x_1^4x_2^4x_3^2x_4^2).
\end{multline*}
\textbf{Substitute $v_k$ for $G$-invariant harmonic polynomials}\\
1. {\it Degree $6$.} 
\begin{align*}
u_6&:=[f_6(v_1'),f_6(v_2'),f_6(v_3'),f_6(v_4')]
=[-1, -\tfrac{1}{9}, \tfrac{1}{9},1].
\end{align*}
2. {\it Degree $8$.} 
\begin{align*}
u_8&:=[f_8(v_1'),f_8(v_2'),f_8(v_3'),f_8(v_4')]
=[1, 
-\tfrac{13}{27}, 
-\tfrac{13}{27},
1].
\end{align*}
3. {\it Degree $12$.} 
\begin{align*}
u_{12,1}&:= [f_{12,1}(v_1'),f_{12,1}(v_2'),f_{12,1}(v_3'),f_{12,1}(v_4')]
=[0, 
\tfrac{128}{243}, 
-\tfrac{25}{243},
1].
\end{align*}
\begin{align*} \label{eq:f122}
u_{12,2}&:=[f_{12,2}(v_1'),f_{12,2}(v_2'),
f_{12,2}(v_3'),f_{12,2}(v_4')]
=[\tfrac{25}{128}, 
\tfrac{1751}{3456}, 0,
1].
\end{align*}

\begin{proposition}
\label{prop:except1}
There is no choice of $R, J$, and $w$ for which
$(\mathcal{X}(F_4,J), w)$ is a Euclidean $12$-design.
\end{proposition}
\begin{proof}
Since we have
\[
-u_{12,1}+2u_{12,2}=[\tfrac{25}{64},\tfrac{7567}{15552},\tfrac{25}{243},1], 
\]
this proposition follows by Lemma~\ref{lem:main_lem}. 
\end{proof}

\subsection{Group $H_3$}
\quad \\
\textbf{Dynkin diagram}
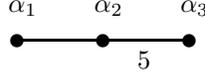
\begin{center}
\unitlength.015in
\begin{picture}(180,25)
\put( 30, 15){\circle*{5}} 
\put( 30, 15){\line(1,0){30}} 
\put( 60, 15){\circle*{5}} 
\put( 60, 15){\line(1,0){30}} 
\put( 90, 15){\circle*{5}} 

\put( 27, 25){$\alpha_1$}
\put( 57, 25){$\alpha_2$}
\put( 87, 25){$\alpha_3$}

\put( 72,5){5}

\end{picture}
\end{center}

\noindent
\textbf{Exponents} 
$\qquad 1,5,9.$

\noindent
\textbf{Fundamental roots} 
\begin{align*}
& \alpha_1:=-{}^te_1+{}^te_2,
\alpha_2:=-{}^te_2+{}^te_3,
\alpha_3:=\tfrac{
(1+\sqrt{2}+\sqrt{5}-\sqrt{10})({}^te_1+{}^te_2)
-(2-\sqrt{2}+2\sqrt{5}+\sqrt{10}) {}^te_3}{6}.
\end{align*}

\noindent
\textbf{Corner Vectors} 
\begin{align*}
& 
v_1=
\tfrac{
-(3 \sqrt{2}+\sqrt{10}+8) {}^te_1
-(3 \sqrt{2}+\sqrt{10}-4) ({}^te_2 + {}^te_3)}{12},\\
& v_2=
\tfrac{
-(3 \sqrt{2}+\sqrt{10}+2) ({}^te_1 + {}^te_2)
-(3 \sqrt{2}+\sqrt{10}-4) {}^te_3}{6},
v_3=
-\tfrac{(\sqrt{2}+\sqrt{10})({}^te_1 + {}^te_2 + {}^te_3)}{4}.
\end{align*}

\noindent
\textbf{Size of Orbit} $\qquad N_1=12, N_2=30, N_3=20.$

\noindent
\textbf{Harmonic Molien series}
\begin{eqnarray*}
\frac{1}{(1-t^6)(1-t^{10})}=
1+t^6+t^{10}+t^{12}+t^{16}+t^{18}+t^{20}+\cdots. 
\end{eqnarray*}

\noindent
\textbf{$G$-invariant harmonic polynomials}\\
For $i=6,10,12$, ${\rm Harm}_i(\mathbb{R}^3)^{H_3}$ is 
spanned by the following: \\
\noindent
1. {\it Degree $6$.} 
\begin{multline*}
f_6:=
2{\rm sym}(x_1^6)
+21{\rm sym}(x_1^5x_2)
-15{\rm sym}(x_1^4x_2^2)
+21\sqrt{10}{\rm sym}(x_1^4x_2x_3)\\
-(70-7\sqrt{10}) {\rm sym}(x_1^3x_2^3)
- 21\sqrt{10}{\rm sym}(x_1^3x_2^2x_3)
+180x_1^2x_2^2x_3^2.
\end{multline*}
\\
\noindent
2. {\it Degree $10$.} 
\[
f_{10}:=\sum_{g \in H_3}h_{10}(x^g),
\]
where
\[
 h_{10}(x):=
256x_1^{10}-5760x_1^8 p_2+20160x_1^6
p_2^2-16800x_1^4 p_2^3+3150x_1^2 p_2^4-63 p_2^5.
\]
\\
\noindent
3. {\it Degree $12$.}
\[
f_{12}:=\sum_{g \in H_3}h_{12}(x^g),
\]
where
\begin{multline*}
h_{12}(x):=1024x_1^{12}-33792x_1^{10} p_2 +190080x_1^8 p_2^2 \\
-295680x_1^6 p_2^3 +138600x_1^4p_2^4 -16632x_1^2 p_2^5+231 p_2^6.
\end{multline*}
\textbf{Substitute $v_k$ for $G$-invariant harmonic polynomials}\\
1. {\it Degree $6$.} 
\begin{align*}
u_6&:=[f_6(v_1'),f_6(v_2'),f_6(v_3') ]
=[\tfrac{14\sqrt{10}-4}{5},
 \tfrac{-7\sqrt{10}+2}{8}, 
 \tfrac{-14\sqrt{10}+4}{9}].
\end{align*}
2. {\it Degree $10$.} 
\begin{align*}
u_{10}&:=[f_{10}(v_1'),f_{10}(v_2'),f_{10}(v_3')] \\
&=[
-\tfrac{43124224\sqrt{10}+49637120}{98415},
\tfrac{8422700\sqrt{10}+9694750}{19683},
-\tfrac{1078105600\sqrt{10}+1240928000}{1594323}].
\end{align*}
3. {\it Degree $12$.} 
\begin{align*}
u_{12}&:=[f_{12}(v_1'),f_{12}(v_2'),f_{12}(v_3')]\\
&=[\tfrac{191679488\sqrt{10}-6897476096}{492075}, 
\tfrac{10856846\sqrt{10}-390677357}{39366},
\tfrac{-191679488\sqrt{10}+6897476096}{14348907}].
\end{align*}

\begin{proposition}
\label{prop:except2}
There is no choice of $R, J$, and $w$ for which
$(\mathcal{X}({H_3}, J), w)$ is a Euclidean $12$-design.
\end{proposition}

\begin{proof}
There is $u\in {\rm Span}_{\mathbb{R}}\{u_6,u_{10},u_{12}\}$
all whose entries are positive,
since the vectors $u_6$, $u_{10}$, $u_{12}$ are linearly independent.
The result follows by Lemma~\ref{lem:main_lem}. 
\end{proof}

\subsection{Group $H_4$}
\quad \\
\textbf{Dynkin diagram}\begin{center}
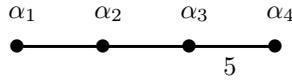
\begin{center}
\unitlength.015in
\begin{picture}(180,25)
\put( 30, 15){\circle*{5}} 
\put( 30, 15){\line(1,0){30}} 
\put( 60, 15){\circle*{5}} 
\put( 60, 15){\line(1,0){30}} 
\put( 90, 15){\circle*{5}} 
\put( 90, 15){\line(1,0){30}} 
\put( 120, 15){\circle*{5}} 

\put( 27, 25){$\alpha_1$}
\put( 57, 25){$\alpha_2$}
\put( 87, 25){$\alpha_3$}
\put( 117, 25){$\alpha_4$}

\put( 102,5){5}

\end{picture}
\end{center}
\end{center}

\noindent
\textbf{Exponents} $\qquad 1,11,19,29.$

\noindent
\textbf{Fundamental roots} 
\begin{align*}
& \alpha_1:= - {}^t e_1 + {}^t e_2,
\alpha_2:= - {}^t e_2 + {}^t e_3,
\alpha_3:=- {}^t e_3 + {}^t e_4,
\alpha_4:= \tfrac{{}^t e_1+{}^te_2+{}^te_3+\sqrt{5} \ {}^te_4}{2}.
\end{align*}

\noindent
\textbf{Corner Vectors} 
\begin{align*}
&
v_1=
\tfrac{(\sqrt{5} - 1) {}^t e_1 + (\sqrt{5} + 3)({}^t e_2 + {}^t e_3 - {}^t e_4)}{4}, 
v_2=
\tfrac{(\sqrt{5} + 1)({}^t e_1 + {}^t e_2)
+ (\sqrt{5} + 3)({}^t e_3 - {}^t e_4)}{2}, \\
& v_3=
\tfrac{(3\sqrt{5} + 5)({}^t e_1 + {}^t e_2 + {}^t e_3)
-3(\sqrt{5} + 3){}^t e_4}{4}, 
v_4=
\tfrac{(\sqrt{5} + 3)({}^t e_1 + {}^t e_2 + {}^t e_3 - {}^t e_4)}{2}.
\end{align*}

\noindent
\textbf{Size of Orbit} $\qquad N_1=120, N_2=720, N_3=1200, N_4=600.$ 

\noindent
\textbf{Harmonic Molien series}
\begin{eqnarray*}
\frac{1}{(1-t^{12})(1-t^{20})(1-t^{30})}=
1+t^{12}+t^{20}+t^{24}+t^{30}+\cdots.
\end{eqnarray*}

\noindent
\textbf{$G$-invariant harmonic polynomials}\\
For $i=12,20,24$, ${\rm Harm}_{i}(\mathbb{R}^4)^{H_4}$ is 
spanned by the following: \\
1. {\it Degree $12$.} \\ 
\[
f_{12}:=\sum_{g \in H_4}h_{12}(x^g),
\]
where
\begin{multline*}
h_{12}(x):=
13x_1^{12}-286x_1^{10} p_3 +1287x_1^8 p_3^2 -1716x_1^6 p_3^3+715
x_1^4 p_3^4 - 78x_1^2 p_3^5+ p_3^6.
\end{multline*}

\noindent
2. {\it Degree $20$.} \\ 
\[
f_{20}:=\sum_{g \in H_4}h_{20}(x^g),
\]
where
\begin{multline*}
h_{20}(x):=21x_1^{20}-1330x_1^{18} p_3 +20349x_1^{16} p_3^2
-116280x_1^{14} p_3^3 + 293930x_1^{12} p_3^4\\
-352716x_1^{10} p_3^5 + 203490x_1^8 p_3^6-54264 x_1^6 p_3^7 + 5985x_1^4 p_3^8
-210x_1^2 p_3^9 + p_3^{10}.
\end{multline*}

\noindent
3. {\it Degree $24$.} \\ 
\[
f_{24}:=\sum_{g \in H_4}h_{24}(x^g),
\]
where 
\begin{multline*}
h_{24}(x):=
x_1^{24}-92x_1^{22} p_3
+ \tfrac{10626}{5}x_1^{20} p_3^2
- 19228x_1^{18} p_3^3 + 81719x_1^{16} p_3^4 \\
- 178296x_1^{14} p_3^5 + 208012x_1^{12} p_3^6
- \tfrac{653752}{5}x_1^{10} p_3^7 + 43263x_1^8 p_3^8 \\
- 7084x_1^6 p_3^9 + 506x_1^4 p_3^{10}
-12x_1^2 p_3^{11} + \tfrac{1}{25} p_3^{12}.
\end{multline*}
\textbf{Substitute $v_k$ for $G$-invariant harmonic polynomials}\\
1. {\it Degree $12$.} 
\begin{align*}
u_{12}&:=[f_{12}(v_1'),f_{12}(v_2'),f_{12}(v_3'),f_{12}(v_4')]
=[-4500, 540, \tfrac{32500}{27}, \tfrac{5625}{4} ].
\end{align*}
2. {\it Degree $20$.} 
\begin{align*}
u_{20}&:=[f_{20}(v_1'), f_{20}(v_2'),f_{20}(v_3'),f_{20}(v_4')]
=[6975, -\tfrac{58869}{25}, \tfrac{4035425}{2187}, \tfrac{216225}{64} ].
\end{align*}
3. {\it Degree $24$.}
\begin{align*}
u_{24}&:=[f_{24}(v_1'),f_{24}(v_2'),f_{24}(v_3'),f_{24}(v_4')]
=[-\tfrac{2367}{16},
-\tfrac{4689027}{50000}, 
\tfrac{416329}{104976},
\tfrac{622521}{16384}].
\end{align*}

\begin{proposition}
\label{prop:except3}
There is no choice of $R, J$, and $w$ for which
$(\mathcal{X}(H_4, J), w)$ is a Euclidean $24$-design.
\end{proposition}

\begin{proof}
Since we have
\[
u_{20}-30u_{24}=[\tfrac{91305}{8},\tfrac{2293281}{5000},
\tfrac{30201755}{17496},\tfrac{18338985}{8192}], 
\]
this proposition follows by Lemma~\ref{lem:main_lem}. 
\end{proof}

\subsection{Group $E_6$}
\quad \\
\textbf{Dynkin diagram}
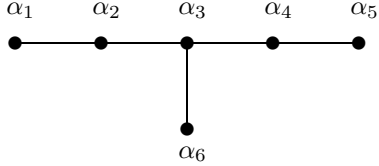
\begin{center}
\unitlength.015in
\begin{picture}(180,45)
\put( 30, 30){\circle*{5}}
\put( 30, 30){\line(1,0){30}}
\put( 60, 30){\circle*{5}}
\put( 60, 30){\line(1,0){30}}
\put( 90, 30){\circle*{5}}
\put( 90, 30){\line(1,0){30}}
\put( 120, 30){\circle*{5}}
\put( 120, 30){\line(1,0){30}}
\put( 150, 30){\circle*{5}}

\put( 90, 0){\circle*{5}}
\put( 90, 0){\line(0,1){30}}

\put( 27, 40){$\alpha_1$}
\put( 57, 40){$\alpha_2$}
\put( 87, 40){$\alpha_3$}
\put( 117, 40){$\alpha_4$}
\put( 147, 40){$\alpha_5$}
\put( 87, -10){$\alpha_6$}
\end{picture}
\end{center}
\quad \\
\noindent
\textbf{Exponents} 
$\qquad 1,4,5,7,8,11.$

\noindent
\textbf{Fundamental roots} 
\begin{align*}
&
\alpha_1:= {}^te_1 - {}^te_2, \alpha_2:= {}^te_2 - {}^te_3,
\alpha_3:= {}^te_3 - {}^te_4, \alpha_4:= {}^te_4 - {}^te_5,
\alpha_5:= {}^te_5 - {}^te_6, \\
&
\alpha_6:=\tfrac{
(-3+\sqrt{3})({}^te_1+{}^te_2+{}^te_3)
+(3+\sqrt{3})({}^te_4+{}^te_5+{}^te_6)
}{6}.
\end{align*}

\noindent
\textbf{Corner Vectors} 
\begin{align*}
& v_1=
\tfrac{
(\sqrt{3} + 5) {}^te_1
+ (\sqrt{3} - 1) ({}^te_2+{}^te_3+{}^te_4+{}^te_5+{}^te_6)}{6},
v_2=
\tfrac{
(\sqrt{3}+2)({}^te_1+{}^te_2)
+ (\sqrt{3}-1)({}^te_3+{}^te_4+{}^te_5+{}^te_6)}{3},\\
& v_3=
\tfrac{(\sqrt{3} + 1)({}^te_1+{}^te_2+{}^te_3)
+ (\sqrt{3} - 1)({}^te_4+{}^te_5+{}^te_6)}{2},
v_4=\tfrac{
(\sqrt{3} + 1)({}^te_1+{}^te_2+{}^te_3+{}^te_4)+
(\sqrt{3} - 2)({}^te_5+{}^te_6)}{3}, \\
& v_5=\tfrac{
(\sqrt{3} + 1)({}^te_1+{}^te_2+{}^te_3+{}^te_4+{}^te_5)+
(\sqrt{3} - 5) {}^te_6}{6},
v_6=\tfrac{\sqrt{3}({}^te_1+{}^te_2+{}^te_3+{}^te_4+{}^te_5+{}^te_6)}{3}.
\end{align*}

\noindent
\textbf{Size of Orbit} $ \qquad 
 N_1=27,
 N_2=216,
 N_3=720,
 N_4=216,
 N_5=27,
 N_6=72.
$

\noindent
\textbf{Harmonic Molien series}
\begin{eqnarray*}
\frac{1}{(1-t^5)(1-t^6)(1-t^8)(1-t^9)(1-t^{12})}=
1+t^5+t^6+t^8+t^9+t^{10}+\cdots. 
\end{eqnarray*}

\noindent
\textbf{$G$-invariant harmonic polynomials}\\
For $i=5,6,8,9,10$, ${\rm Harm}_i(\mathbb{R}^6)^{E_6}$ 
is spanned by the following: \\
1. {\it Degree $5$.} 
\begin{multline*}
f_5:={\rm sym}(x_1^5)+{\rm sym}(x_1^4 x_2 )-2 {\rm sym}(x_1^3 x_2^2 )
+{\rm sym}(x_1^3 x_2 x_3 )\\-3{\rm sym}(x_1^2 x_2 x_3  x_4)+24{\rm sym}(x_1 x_2 x_3  x_4 x_5).
\end{multline*}
\noindent
2. {\it Degree $6$.} 
\begin{multline*}
f_6:={\rm sym}(x_1^6)+ \tfrac{3}{2}{\rm sym}(x_1^5 x_2)+ 
-3{\rm sym}(x_1^4 x_2^2)+\tfrac{15}{14}{\rm sym}(x_1^4 x_2 x_3) \\
+\tfrac{5}{7}{\rm sym}(x_1^3 x_2^3)
-\tfrac{30}{7}{\rm sym}(x_1^3 x_2^2 x_3)
+\tfrac{30}{7}{\rm sym}(x_1^3 x_2 x_3 x_4)
+9{\rm sym}(x_1^2 x_2^2 x_3^2) \\
+\tfrac{45}{7}{\rm sym}(x_1^2 x_2^2 x_3 x_4)
-\tfrac{180}{7}{\rm sym}(x_1^2 x_2 x_3 x_4 x_5)
+\tfrac{180}{7}x_1 x_2 x_3 x_4 x_5 x_6.
\end{multline*}
\noindent
3. {\it Degree $8$.} 
\[
f_{8}:=\sum_{g \in E_6}h_8(x^g),
\]
where 
\begin{align*}
& h_8(x):=
x_1^8-\tfrac{28}{5}x_1^6 p_5 + 6x_1^4 p_5^2
-\tfrac{4}{3}x_1^2 p_5^3 + \tfrac{1}{33} p_5^4.
\end{align*}

\noindent
4. {\it Degree $9$.} 
 \[
f_{9}:=\sum_{g \in E_6}h_9(x^g),
\]
where 
\begin{multline*}
h_9(x):=
{\rm sym}(x_1^9) - \tfrac{36}{5}{\rm sym}(x_1^7x_2^2) + \tfrac{126}{5}{\rm sym}(x_1^5x_2^4) - 63{\rm sym}(x_1^4x_2^3x_3^2) \\
+ 63{\rm sym}(x_1^4x_2^2x_3^2x_4) + 252{\rm sym}(x_1^3x_2^2x_3^2x_4^2)-945{\rm sym}(x_1^2x_2^2x_3^2x_4^2x_5).
\end{multline*}

\noindent
5. {\it Degree $10$.} 
\[
f_{10}:=\sum_{g \in E_6}h_{10}(x^g),
\]
where 
\begin{align*}
& h_{10}(x):=
x_1^{10}-9x_1^8 p_5 + 18x_1^6 p_5^2 - 10x_1^4 p_5^3
+\tfrac{15}{11}x_1^2 p_5^4 - \tfrac{3}{143} p_5^5.
\end{align*}

\noindent
\textbf{Substitute $v_k$ for $G$-invariant harmonic polynomials}\\
1. {\it Degree $5$.}
\begin{align*}
u_5&:=[f_5(v_1'),f_5(v_2'),f_5(v_3'),f_5(v_4'),f_5(v_5'),f_5(v_6')] \\
& =[\tfrac{3 \sqrt{3}}{4},\tfrac{6 \sqrt{30}}{125},0,-\tfrac{6 \sqrt{30}}{125},-\tfrac{3 
\sqrt{3}}{4},0].
\end{align*}
2. {\it Degree $6$.} 
\begin{align*}
u_6&:=[f_6(v_1'),f_6(v_2'),f_6(v_3'),f_6(v_4'),f_6(v_5'),f_6(v_6')] \\
& =[\tfrac{81}{56},-\tfrac{81}{700},-\tfrac{9}{28},-\tfrac{81}{700},\tfrac{81}{56},-\tfrac{27}{28}].
\end{align*}
3. {\it Degree $8$.} 
\begin{align*}
u_8&:=[f_8(v_1'),f_8(v_2'),f_8(v_3'),f_8(v_4'),f_8(v_5'),f_8(v_6')]
\\ &
=[800,-\tfrac{6784}{25},-\tfrac{640}{9},-\tfrac{6784}{25},800,\tfrac{3200}{3}].
\end{align*}
4. {\it Degree $9$.} 
\begin{align*}
u_9&:=[f_9(v_1'),f_9(v_2'),f_9(v_3'),f_9(v_4'),f_9(v_5'),f_9(v_6')]
\\ &=[2065 \sqrt{3},-\tfrac{185024 \sqrt{30}}{625},0,\tfrac{185024 \sqrt{30}}{625},-
2065 \sqrt{3},0].
\end{align*}
5. {\it Degree $10$.} 
\begin{align*}
u_{10}&:=[f_{10}(v_1'),f_{10}(v_2'),f_{10}(v_3'),f_{10}(v_4'),f_{10}(v_5'),f_{10}(v_6')]\\
&=[\tfrac{11520}{13},\tfrac{423936}{1625},\tfrac{51200}{351},\tfrac{423936}{1625},\tfrac{11520}{13},-\tfrac{10240}{39}].
\end{align*}

\begin{proposition}
\label{prop:except4}
There is no choice of $R, J$, and $w$ for which
$(\mathcal{X}({E_6}, J), w)$ is a Euclidean $10$-design.
\end{proposition}
\begin{proof}
Since we have
\[
u_{10}+u_{8}=[\tfrac{11745}{2816},\tfrac{13527}{220000},
\tfrac{387}{1760},\tfrac{13527}{220000},\tfrac{11745}{2816},\tfrac{621}{352}], 
\]
this proposition follows by Lemma~\ref{lem:main_lem}. 
\end{proof}

\subsection{Group $E_7$}
\quad \\
\textbf{Dynkin diagram}
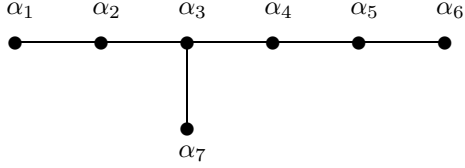
\begin{center}
\unitlength.015in
\begin{picture}(180,45)
\put( 30, 30){\circle*{5}}
\put( 30, 30){\line(1,0){30}}
\put( 60, 30){\circle*{5}}
\put( 60, 30){\line(1,0){30}}
\put( 90, 30){\circle*{5}}
\put( 90, 30){\line(1,0){30}}
\put( 120, 30){\circle*{5}}
\put( 120, 30){\line(1,0){30}}
\put( 150, 30){\circle*{5}}
\put( 150, 30){\line(1,0){30}}
\put( 180, 30){\circle*{5}}

\put( 90, 0){\circle*{5}}
\put( 90, 0){\line(0,1){30}}

\put( 27, 40){$\alpha_1$}
\put( 57, 40){$\alpha_2$}
\put( 87, 40){$\alpha_3$}
\put( 117, 40){$\alpha_4$}
\put( 147, 40){$\alpha_5$}
\put( 177, 40){$\alpha_6$}
\put( 87, -10){$\alpha_7$}
\end{picture}
\end{center}
\qquad \\
\noindent
\textbf{Exponents} 
$\qquad 1,5,7,9,11,13,17.$

\noindent
\textbf{Fundamental roots} 
\begin{align*}
& \alpha_1:= {}^te_1-{}^te_2, \alpha_2:= {}^te_2-{}^te_3,
  \alpha_3:= {}^te_3-{}^te_4, \alpha_4:= {}^te_4-{}^te_5,
  \alpha_5:= {}^te_5-{}^te_6, \\
& \alpha_6:= {}^te_6-{}^te_7, \alpha_7:=\tfrac{
(-4+\sqrt{2})({}^te_1+{}^te_2+{}^te_3)+
(3+\sqrt{2})({}^te_4+{}^te_5+{}^te_6+{}^te_7)}{7}.
\end{align*}

\noindent
\textbf{Corner Vectors} 
\begin{align*}
& v_1=
\tfrac{
(6+2 \sqrt{2}) {}^te_1 +
(-1+2 \sqrt{2}) ({}^te_2 + {}^te_3 + {}^te_4 + {}^te_5 + {}^te_6 + {}^te_7)}{7}, \\
& v_2=\tfrac{
(5+4 \sqrt{2})({}^te_1 + {}^te_2) +
(-2+4 \sqrt{2})({}^te_3 + {}^te_4 + {}^te_5 + {}^te_6 + {}^te_7)}{7}, \\
& v_3=\tfrac{
(4+6 \sqrt{2})({}^te_1 + {}^te_2 + {}^te_3)+
(-3+6 \sqrt{2}) ({}^te_4 + {}^te_5 + {}^te_6 + {}^te_7)}{7}, \\
& v_4=\tfrac{
(6+9 \sqrt{2})({}^te_1 + {}^te_2 + {}^te_3 + {}^te_4)+
(-8+9 \sqrt{2})({}^te_5 + {}^te_6 + {}^te_7)}{14},\\
& v_5=\tfrac{
(2+3 \sqrt{2})({}^te_1 + {}^te_2 + {}^te_3 + {}^te_4 + {}^te_5)+
(-5+3 \sqrt{2})({}^te_6 + {}^te_7)}{7}, \\
& v_6=\tfrac{
(-2-3 \sqrt{2})({}^te_1 + {}^te_2 + {}^te_3 + {}^te_4 + {}^te_5 + {}^te_6)+
(12-3 \sqrt{2}){}^te_7}{14},\\
& v_7=\tfrac{{}^te_1 + {}^te_2 + {}^te_3 + {}^te_4 + {}^te_5 + {}^te_6 + {}^te_7}{\sqrt{2}}.
\end{align*}

\noindent
\textbf{Size of Orbit}
\[N_1=126, N_2=2016, N_3=10080, N_4=4032, N_5=756, N_6=56, N_7=576.
\]

\noindent
\textbf{Harmonic Molien series}
\begin{eqnarray*}
\frac{1}{(1-t^6)(1-t^8)(1-t^{10})(1-t^{12})(1-t^{14})(1-t^{18})}=
1+t^6+t^8+t^{10}+2t^{12}+
\cdots. 
\end{eqnarray*}

\noindent
\textbf{$G$-invariant harmonic polynomials}\\
For $i=6,8,10,12$, ${\rm Harm}_i(\mathbb{R}^7)^{E_7}$ is 
spanned by the following: \\
1. {\it Degree $6$.} 
\[
f_{6}:=\sum_{g \in E_7}h_6(x^g),
\]
where 
\begin{align*}
h_6(x):=32x_1^6-80x_1^4 p_6 + 30x_1^2 p_6^2 - p_6^3.
\end{align*}

\noindent
2. {\it Degree $8$.} 
\[
f_{8}:=\sum_{g \in E_7}h_8(x^g),
\]
where 
\begin{align*}
h_8(x):=384x_1^8-1792x_1^6 p_6 +1680x_1^4 p_6^2 - 336x_1^2 p_6^3 + 7 p_6^4.
\end{align*}

\noindent
3. {\it Degree $10$.} 
\[
f_{10}:=\sum_{g \in E_7}h_{10}(x^g),
\]
where 
\begin{align*}
h_{10}(x):=256x_1^{10}-1920x_1^8 p_6 + 3360x_1^6 p_6^2
- 1680x_1^4 p_6^3 + 210x_1^2 p_6^4 - 3 p_6^5.
\end{align*}

\noindent
4. {\it Degree $12$.} 
\begin{align*}
& f_{12,1}:=\sum_{g \in E_7}h_{12,1}(x^g),
\quad f_{12,2}:=\sum_{g \in E_7}h_{12,2}(x^g),
\end{align*}
where 
\begin{multline*}
h_{12,1}(x):=
4096x_1^{12}-45056x_1^{10} p_6 +126720x_1^8 p_6^2 \\
-118272x_1^6 p_6^3 + 36960x_1^4 p_6^4 - 3168x_1^2 p_6^5+33 p_6^6,
\end{multline*}
\begin{multline*}
h_{12,2}(x):=
x_1x_2(2048x_1^{10}-14080x_1^8 p_6 + 25344x_1^6 p_6^2
- 14784x_1^4 p_6^3+ 2640x_1^2 p_6^4- 99 p_6^5).
\end{multline*}

\noindent
\textbf{Substitute $v_k$ for $G$-invariant harmonic polynomials}\\
1. {\it Degree $6$.} 
\begin{align*}
u_6&:=[f_6(v_1'),f_6(v_2'),f_6(v_3'),f_6(v_4'),f_6(v_5'),f_6(v_6')
,f_6(v_7')]\\
&=[\tfrac{-7700659200+9488793600 \sqrt{2}}{16807},
\tfrac{-427814400+527155200 \sqrt{2}}{2401},
\tfrac{-1818211200+2240409600 \sqrt{2}}{16807}, \\
&\qquad
\tfrac{-547602432+674758656 \sqrt{2}}{16807},
\tfrac{2887747200-3558297600 \sqrt{2}}{16807},
\tfrac{20535091200-25303449600 \sqrt{2}}{16807}, \\
&\qquad
\tfrac{-123210547200+151820697600 \sqrt{2}}{823543}].
\end{align*}
2. {\it Degree $8$.} 
\begin{align*}
u_8&:=[f_8(v_1'),f_8(v_2'),f_8(v_3'),f_8(v_4'),f_8(v_5'),f_8(v_6'),f_8(v_7')]\\
&=[\tfrac{6579988992000-5480856576000 \sqrt{2}}{823543},
\tfrac{731109888000-608984064000 \sqrt{2}}{823543},\\
&\qquad
\tfrac{-3527605209600+2938348108800 \sqrt{2}}{823543}, 
\tfrac{-3134999199744+2611323666432 \sqrt{2}}{823543},\\
&\qquad
\tfrac{-1809496972800+1507235558400 \sqrt{2}}{823543},
\tfrac{3509327462400-2923123507200 \sqrt{2}}{823543}, \\
&\qquad
\tfrac{-115807806259200+96463075737600 \sqrt{2}}{40353607}].
\end{align*}
3. {\it Degree $10$.} 
\begin{align*}
u_{10}&:=[f_{10}(v_1'),f_{10}(v_2'),f_{10}(v_3'),f_{10}(v_4'),f_{10}(v_5'),f_{10}(v_6'),f_{10}(v_7')]\\
&=[\tfrac{-6428624451840+415928908800 \sqrt{2}}{5764801},
\tfrac{357145802880-23107161600 \sqrt{2}}{823543},\\
&\qquad
\tfrac{2388412556760-154529143200 \sqrt{2}}{5764801},
\tfrac{-73143460429824+4732346695680 \sqrt{2}}{720600125},\\
&\qquad\tfrac{-7433097022440+480917800800 \sqrt{2}}{5764801},
\tfrac{30476441845760-1971811123200 \sqrt{2}}{5764801}, \\
&\qquad\tfrac{3291455719342080-212955601305600 \sqrt{2}}{1977326743}].
\end{align*}
4. {\it Degree $12$.} 
\begin{align*}
u_{12,1}&:=[f_{12,1}(v_1'),f_{12,1}(v_2'),f_{12,1}(v_3'),f_{12,1}(v_4'),f_{12,1}(v_5'),f_{12,1}(v_6'),f_{12,1}(v_7')]\\
&=[\tfrac{27363005574796800+17942314142016000 \sqrt{2}}{1977326743},
\tfrac{-760132890073600-689327933184000 \sqrt{2}}{282475249},\\ 
&\qquad\tfrac{-513174301527400-792264524693400 \sqrt{2}}{1977326743},
\tfrac{-14026148038967296-72141536776421376 \sqrt{2}}{49433168575},\\ 
&\qquad\tfrac{3931481294451000-4960153279164600 \sqrt{2}}{1977326743},
\tfrac{-12979679661260800-37832882828083200 \sqrt{2}}{1977326743},\\ 
&\qquad\tfrac{249757080640811827200+284898146732782387200 \sqrt{2}}{33232930569601}].
\end{align*}
\begin{align*}
u_{12,2}&:=[f_{12,2}(v_1'),f_{12,2}(v_2'),f_{12,2}(v_3'),f_{12,2}(v_4'),f_{12,2}(v_5'),f_{12,2}(v_6'),f_{12,2}(v_7')]\\
&=[\tfrac{-2419675164360000-1489162193640000 \sqrt{2}}{1977326743},
\tfrac{113867977056000+18727597152000 \sqrt{2}}{282475249},\\ 
&\qquad\tfrac{156757191916575-26126480038725 \sqrt{2}}{1977326743},
\tfrac{16622260339703808-6701937797136384 \sqrt{2}}{49433168575},\\ 
&\qquad\tfrac{1494452243214675-1107985853945025 \sqrt{2}}{1977326743},
\tfrac{8313279170969600-2771226582220800 \sqrt{2}}{1977326743},\\ 
&\qquad\tfrac{-51686387833407897600+773622407142604800 \sqrt{2}}{33232930569601}].
\end{align*}

\begin{proposition}
\label{prop:except5}
There is no choice of $R, J$, and $w$ for which
$(\mathcal{X}(E_7, J), w)$ is a Euclidean $12$-design.
\end{proposition}

\begin{proof}
Since we have
\begin{multline*}
-2u_{12,1}-25u_{12,2}+u_{10}=\\
[2.86443\times 10^6,256489.,513956.,989894.,2.86352\times 10^6,1.64917\times 10^7,293023.], 
\end{multline*}
this proposition follows by Lemma~\ref{lem:main_lem}. 
\end{proof}

\subsection{Group $E_8$} \label{subsec:E8}
\quad \\
\textbf{Dynkin diagram}\begin{center}
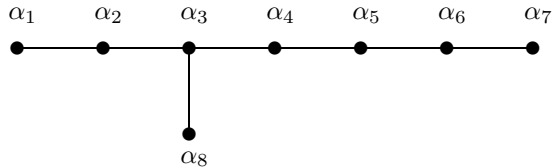
\begin{center}
\unitlength.015in
\begin{picture}(180,45)
\put( 30, 30){\circle*{5}}
\put( 30, 30){\line(1,0){30}}
\put( 60, 30){\circle*{5}}
\put( 60, 30){\line(1,0){30}}
\put( 90, 30){\circle*{5}}
\put( 90, 30){\line(1,0){30}}
\put( 120, 30){\circle*{5}}
\put( 120, 30){\line(1,0){30}}
\put( 150, 30){\circle*{5}}
\put( 150, 30){\line(1,0){30}}
\put( 180, 30){\circle*{5}}
\put( 180, 30){\line(1,0){30}}
\put( 210, 30){\circle*{5}}

\put( 90, 0){\circle*{5}}
\put( 90, 0){\line(0,1){30}}

\put( 27, 40){$\alpha_1$}
\put( 57, 40){$\alpha_2$}
\put( 87, 40){$\alpha_3$}
\put( 117, 40){$\alpha_4$}
\put( 147, 40){$\alpha_5$}
\put( 177, 40){$\alpha_6$}
\put( 207, 40){$\alpha_7$}
\put( 87, -10){$\alpha_8$}
\end{picture}
\end{center}
\end{center}
\qquad \\
\noindent
\textbf{Exponents} 
$\qquad 1,7,11,13,17,19,23,29.$

\noindent
\textbf{Fundamental roots} 
\begin{align*}
& \alpha_1:= {}^t e_1 - {}^t e_2,
\alpha_2:= {}^t e_2 - {}^t e_3,
\alpha_3:= {}^t e_3 - {}^t e_4,
\alpha_4:= {}^t e_4 - {}^t e_5,
\alpha_5:= {}^t e_5 - {}^t e_6, \\
& \alpha_6:= {}^t e_6 - {}^t e_7,
\alpha_7:={}^t e_7 - {}^t e_8,
\alpha_8:=\tfrac{-{}^t e_1 -{}^t e_2 -{}^t e_3
+ {}^t e_4 + {}^t e_5 + {}^t e_6 + {}^t e_7 + {}^t e_8}{2}.
\end{align*}

\noindent
\textbf{Corner Vectors} 
\begin{align*}
&v_1=
\tfrac{3 {}^te_1+{}^te_2+{}^te_3+{}^te_4+{}^te_5+{}^te_6+{}^te_7+{}^te_8}{2}, \\
&v_2=
2 {}^te_1+ 2 {}^te_2+{}^te_3+{}^te_4+{}^te_5+{}^te_6+{}^te_7+{}^te_8,\\
&v_3=
\tfrac{5 {}^te_1+ 5 {}^te_2+ 5 {}^te_3+3 {}^te_4+3 {}^te_5+3 {}^te_6+3 {}^te_7+3 {}^te_8}{2}, \\
&v_4=
2 {}^te_1+ 2 {}^te_2+ 2 {}^te_3+ 2 {}^te_4+{}^te_5+{}^te_6+{}^te_7+{}^te_8,\\
&v_5=
\tfrac{3 {}^te_1+ 3 {}^te_2+ 3 {}^te_3+ 3 {}^te_4+ 3 {}^te_5+{}^te_6+{}^te_7+{}^te_8}{2}, \\
&v_6=
-{}^te_1-{}^te_2-{}^te_3-{}^te_4-{}^te_5-{}^te_6,\\
&v_7=
\tfrac{-{}^te_1-{}^te_2-{}^te_3-{}^te_4-{}^te_5-{}^te_6-{}^te_7+{}^te_8}{2}, \\
&v_8=
{}^te_1+{}^te_2+{}^te_3+{}^te_4+{}^te_5+{}^te_6+{}^te_7+{}^te_8.
\end{align*}

\noindent
\textbf{Size of Orbit}
\[
N_1=2160, 
N_2=69120, 
N_3=483840, 
N_4=241920, \]
\[N_5=60480, 
N_6=6720, 
N_7=240, 
N_8=17280.
\]

\noindent
\textbf{Harmonic Molien series}
\begin{multline*}
\frac{1}{(1-t^8)(1-t^{12})(1-t^{14})(1-t^{18})(1-t^{20})
(1-t^{24})(1-t^{30})}=\\
1+t^8+t^{12}+t^{14}+t^{16}+t^{18}+2 t^{20}+
\cdots. 
\end{multline*}

\noindent
\textbf{$G$-invariant harmonic polynomials}\\
For $i=8,12,14,16$, ${\rm 
Harm}_8(\mathbb{R}^8)^{E_8}$ is spanned by the following: \\
1. {\it Degree $8$.} 
\[
f_{8}:=\sum_{g \in E_8}h_8(x^g),
\]
where 
\begin{align*}
h_8(x):=
429x_1^8-1716x_1^6 p_7 + 1430x_1^4 p_7^2 - 260x_1^2 p_7^3 + 5 p_7^4.
\end{align*}

\noindent
2. {\it Degree $12$.} 
\[
f_{12}:=\sum_{g \in E_8}h_{12}(x^g),
\]
where 
\begin{multline*}
h_{12}(x):=1547x_1^{12}-14586x_1^{10} p_7 + 36465x_1^8 p_7^2
- 30940x_1^6 p_7^3 + 8925x_1^4 p_7^4 - 714x_1^2 p_7^5 + 7 p_7^6.
\end{multline*}

\noindent
3. {\it Degree $14$.} 
\[
f_{14}:=\sum_{g \in E_8}h_{14}(x^g),
\]
where 
\begin{multline*}
h_{14}(x):=969x_1^{14}-12597x_1^{12} p_7 + 46189x_1^{10} p_7^2
- 62985x_1^8 p_7^3 \\
+ 33915x_1^6 p_7^4 - 6783x_1^4 p_7^5 + 399x_1^2 p_7^6 - 3 p_7^7.
\end{multline*}

\noindent
4. {\it Degree $16$.} 
\[
f_{16}:=\sum_{g \in E_8}h_{16}(x^g),
\]
where 
\begin{multline*}
h_{16}(x):=6783x_1^{16}- 116280x_1^{14} p_7 + 587860 x_1^{12} p_7^2
-1175720x_1^{10} p_7^3 \\
+ 1017450 x_1^8 p_7^4 -379848x_1^6 p_7^5 + 55860x_1^4 p_7^6
-2520x_1^2 p_7^7 +15 p_7^8.
\end{multline*}

\noindent
\textbf{Substitute $v_k$ for $G$-invariant harmonic polynomials}\\
1. {\it Degree $8$.} 
\begin{align*}
u_8&:=[f_8(v_1'),f_8(v_2'),f_8(v_3'),f_8(v_4'),f_8(v_5'),f_8(v_6'),f_8(v_7'),f_8(v_8')]\\
&=[174182400, \tfrac{4926873600}{49}, 82059264, 62705664, \\
&\qquad 19353600, -116121600,
-1045094400, 97977600 ].
\end{align*}
2. {\it Degree $12$.} 
\begin{align*}
u_{12}&:=[f_{12}(v_1'),f_{12}(v_2'),f_{12}(v_3'),f_{12}(v_4'),f_{12}(v_5'),f_{12}(v_6'),f_{12}(v_7'),f_{12}(v_8')]\\
&=[1680315840, \tfrac{15655887360}{49}, \tfrac{14950365696}{125}, -\tfrac{2608490304}{125}, \\ 
&\qquad -275607360,
-734952960, 4480842240, 148777965 ].
\end{align*}
3. {\it Degree $14$.} 
\begin{align*}
u_{14}&:=[f_{14}(v_1'),f_{14}(v_2'),f_{14}(v_3'),f_{14}(v_4'),f_{14}(v_5'),f_{14}(v_6'),f_{14}(v_7'),f_{14}(v_8')]\\
&=[1207483200, -\tfrac{567924825600}{16807}, -\tfrac{2009165312}{15}, -\tfrac{671799744}{5},\\
&\qquad -\tfrac{253422400}{3},
184307200, -2634508800, -293294925].
\end{align*}
4. {\it Degree $16$.} 
\begin{align*}
u_{16}&:=[f_{16}(v_1'),f_{16}(v_2'),f_{16}(v_3'),f_{16}(v_4'),f_{16}(v_5'),f_{16}(v_6'),f_{16}(v_7'),f_{16}(v_8')]\\
&=[1490121360, -\tfrac{393199971840}{2401}, \tfrac{3287394820358656}{1265625}, \tfrac{36512571016971}{62500}, \\
&\qquad    \tfrac{1232569520}{3},2075906560, 7529034240,
 -\tfrac{9749511135}{16}].
\end{align*}

\begin{proposition}
\label{prop:except6}
There is no choice of $R, J$, and $w$ for which
$(\mathcal{X}(E_8, J), w)$ is a Euclidean $16$-design.
\end{proposition}
\begin{proof}
Since we have
\begin{multline*}
u_{16}-3u_{14}+2u_{12} 
=
[1228303440,\tfrac{9691313402880}{16807},\tfrac{4098709695302656}{1265625}, 
\tfrac{59096571112971}{62500}, \\
\tfrac{339192560}{3},53079040,24394245120,\tfrac{9089540145}{16}], 
\end{multline*}
this proposition follows by Lemma~\ref{lem:main_lem}. 
\end{proof}

\bigskip

Now, we are ready to complete the proof of Theorem~\ref{thm:NS}.

\bigskip
\noindent
{\it Proof of Theorem~\ref{thm:NS}}:
The case (1) is in Theorem~\ref{thm:Bajnok},
and the cases (2), (3) in~\cite{NS11}.
Thus the theorem follows by
Propositions~\ref{prop:except1}-\ref{prop:except6}.
$\Box$

\bigskip

The following result, together with Theorem~\ref{thm:NS},
determine the maximum degree of spherical cubature formulae
$(\mathcal{X}(G, J), w)$ for all irreducible reflection groups $G$.

\begin{theorem}
\label{thm:NS2}
(i) An $F_4$-invariant cubature of degree $11$
that consists of the orbits of the corner vectors
is classified by:
\begin{align*}
& w_1 = \tfrac{13-960 w_4}{960}, \quad w_2 =\tfrac{3(-1+192 w_4)}{256},
\quad w_3 = \tfrac{3(1-120 w_4)}{160}, \quad \tfrac{1}{192} \leq w_4 \leq \tfrac{1}{120}.
\end{align*}
(ii) An $H_3$-invariant cubature of degree $11$
that consists of the orbits of the corner vectors
is classified by:
\begin{align*}
w_1=\tfrac{125}{5544}, \quad w_2=\tfrac{64}{3465}, \quad w_3=\tfrac{27}{3080}.
\end{align*}
(iii) An $H_4$-invariant cubature of degree $23$
that consists of the orbits of the corner vectors
is classified by:
\begin{align*}
& w_1=\tfrac{368-9625 w_4}{315392}, \quad w_2=\tfrac{125 (16+5625 w_4)}{2359296},
\quad w_3= -\tfrac{6561 (16-51975 w_4)}{504627200}, \quad 0\leq w_4 \leq \tfrac{16}{51975}.
\end{align*}
(iv) An $E_6$-invariant cubature of degree $9$
that consists of the orbits of the corner vectors
is classified by:
\begin{align*}
& w_1= \tfrac{2(1-96 w_6)}{729}, \quad w_2= \tfrac{125 (1+1200 w_6)}{186624},\quad
w_3= \tfrac{1}{1280}-\tfrac{9 w_6}{16}, \\
& w_4=\tfrac{125 (1+1200 w_6)}{186624}, \quad
w_5= \tfrac{2(1-96 w_6)}{729}, \quad 0\leq w_6 \leq \tfrac{1}{720}.
\end{align*}
(v) An $E_7$-invariant cubature of degree $11$
that consists of the orbits of the corner vectors
is classified by the following two types of weights:
\begin{enumerate}
\item 
$ w_1= -\tfrac{4 (-296924467+966078461040 w_2+107900687895 w_3+95875084800 w_7)}{610410794301}$, \\
$w_4= -\tfrac{625 (-945994+3215011030 w_2+24066363475 w_3+1769169600 w_7)}{4340698981696}$, \\
$ w_5= \tfrac{8 (34900936+247702641648 w_2+1231161574335 w_3+182083866624 w_7)}{1831232382903}$,\\
$w_6=-\tfrac{27 (-32430307+60983896974 w_2+30607311735 w_3+25518620160 w_7)}{542587372712}$,\\
$0\leq w_7\leq -\tfrac{2401 (-394+1339030 {w_2}+10023475 {w_3})}{1769169600}$, $0\leq w_3<-\tfrac{2 (-197+669515 w_2)}{10023475}$, \\
$ 0\leq w_2\leq \tfrac{197}{669515}$. 
\item 
$ w_1= -\tfrac{4 (-211+686070 w_2)}{440055}$, 
$w_3=-\tfrac{2 (-197+669515 w_2)}{10023475}$,  
$w_5= \tfrac{16 (1231+1230075 w_2)}{54126765}$, \\
$w_6=-\tfrac{351 (-71+129360 w_2)}{16037560}$,
 $0\leq w_2\leq \tfrac{197}{669515}$, 
$ w_4= 0$, 
 $ w_7=0$.
\end{enumerate}
(vi) An $E_8$-invariant cubature of degree $15$
that consists of the orbits of the corner vectors
is classified by the nonnegative solutions $w_i$ of the system of equations 
\begin{equation} \label{eq:weight_E8}
u_8 {}^tv=0, \qquad u_{12} {}^tv=0, \qquad  u_{14} {}^tv=0, 
\qquad \sum_{i=1}^8 N_i w_i=1,
\end{equation}
where $v=(N_1 w_1, \ldots, N_8 w_8)$, 
and $u_i$, $N_i$ are defined in Subsection~\ref{subsec:E8}. 
The precise solutions of \eqref{eq:weight_E8} are referred to Appendix. 
\end{theorem}

\begin{remark}
The $H_3$-invariant cubature of Theorem~\ref{thm:NS2}~(i)
was constructed by Goethals and Seidel~\cite[p.~214]{GS82} who
found, moreover, a spherical cubature of degree $15$ by
taking the orbits of $v_1', v_2', v_3'$, plus one more orbit;
for example, see~\cite{HS96}
for further informations on the existence of
three-dimensional spherical cubature.
It is also interesting to note that
the formula given in Theorem~\ref{thm:NS2}~(vi)
is equivalent to a $26400$-point cubature of degree $15$ which
comes from shells of the Korkin-Zorotalev lattice~\cite{HPV07}.
In~\cite[p.~214]{GS82},
Goethals and Seidel found a spherical cubature of degree $19$ that
consists of the $H_4$-orbits of the zeros of
an invariant harmonic homogeneous polynomial of degree $12$.
Salihov~\cite{S75} found another $H_4$-invariant cubature of
degree $19$ by taking the union of the $120$-cell and the $600$-cell.
Motivated by this,
the authors searched three and four
$H_4$-orbits of the corner vectors,
and found the higher-degree cubature of Theorem~\ref{thm:NS2}~(ii).
\end{remark}

\section{Hilbert identities and cubature formulae}

As explained in Section 2,
there is a cubature of index $q$ on $S^{m-1}$ with $n$ points if
and only if
there are $n$ vectors $r_1, \ldots, r_n \in \mathbb{R}^m$ such that
\begin{equation}
\label{eq:identity}
\sum_{i=1}^n \langle x, r_i \rangle^q
= \langle x, x \rangle^\frac{q}{2}
\end{equation}
for every $x \in \mathbb{R}^m$.
Identity (\ref{eq:identity}) yields
a representation of $(\sum_{i=1}^m x_i^2 )^{q/2}$
as a sum of $q$th powers of real linear forms
with positive real coefficients.
Such a representation
is called a {\it Hilbert identity}~\cite{R11}.
Various aesthetic meanings of Hilbert identities are extensively discussed
in a famous paper by Reznick~\cite{R92}.

Many Hilbert identities can be obtained
by the cubature that are constructed in Sections 4 and 5.
In particular, some of the resulting identities are
represented as a sum of $q$th powers of rational linear forms
with positive rational coefficients.
Such {\it rational representations} were used not only in studying
Waring's problem~\cite[pp.~717-725]{D23},
but also in the work of Schmid on real holomorphy rings~\cite{S94}.
An aesthetic meaning of rational representations
would be stated as follows\footnote{This was suggested by Bruce Reznick
through email conversation.}:
We would take all coefficients $\{a_i\}$ which
appear in a formula, and consider the field created by
adjoining them, and then look at its dimension
$[\mathbb{Q}(\{a_i\}) : \mathbb{Q}]$.
With this measure,
the ^^ ^^ best formulas" would only involve rationals,
and the minimum value occurs if the coefficients
are already in $\mathbb{Q}$.

It is well known (it goes back to Hilbert~\cite{H09}) that
\begin{equation}
\label{eq:identity2}
\int_{S^{m-1}} y_1^q \rho ({\rm d} y) =  \frac{(q-1)!!(m-2)!!}{(m+q-2)!!}.
\end{equation}
This is certainly a rational number.
All cubature given in Section 4 have rational weights, and
points from orbits of the form
$(\sqrt[q]{a}, \ldots, \sqrt[q]{a}, 0, \ldots, 0)^{B_m}$ with
rational $a$.
Thus, by Proposition~\ref{prop:SSI},
we can obtain many rational representations.

For example,
the $91$-point cubature of Example~\ref{exam:SXind61}
is translated into the following rational representation that
Reznick~\cite{R92} was not able to find.
\begin{theorem}
\label{thm:identity}
\begin{align}
\label{eq:Sawaidentity2}
120 (\sum_{i=1}^7 x_i^2)^3
& = \sum_{56} ( x_i \pm x_{i+2} \pm x_{i+3} \pm x_{i+4})^6
+ 2 \sum_{28} ( x_i \pm x_{i+2} \pm x_{i+3})^6
+ \sum_7 (2x_i)^6
\end{align}
where on the right the indices are taken as cyclic modulo $7$
and all possible combinations of signs occur in the summation.
\end{theorem}
\begin{remark}
Reznick~\cite[p.~112]{R92} translated
an index-six cubature on $S^6$ which was found by Stroud in 1967
into the following beautiful representation:
\begin{align}
\label{eq:Reznick1}
960 (\sum_{i=1}^7 x_i^2)^3
= 2 \sum_7 (2x_i)^6 + \sum_{2 \cdot \binom{7}{2}} (2x_i \pm 2x_j)^6
+ \sum_{2^6} (x_1 \pm \cdots \pm x_7)^6,
\end{align}
where on the right
all possible combinations of signs and
pairs of the $7$ variables $x_1, \ldots, x_7$ occur
in the second summation.
Identity (\ref{eq:Sawaidentity2}) improves Reznick's representation.
Namely,
(\ref{eq:Sawaidentity2}) has
fewer number of sixth powers
than (\ref{eq:Reznick1}).
\end{remark}

More rational representations are available.
For example, look at the following K\"ursch\'ak's representation:
\begin{align*}
2^k \binom{3k}{k} (\sum_{i=1}^{3k+1} x_i^2)^2 &=
\sum (x_{i_1} \pm x_{i_2} \pm \ldots \pm x_{i_{k+1}})^4
\end{align*}
where on the right all possible combinations of signs and
$(k+1)$-subsets of the $3k+1$ variables $x_1, \ldots, x_{3k+1}$
occur~\cite[p.~723]{D23}.
This corresponds to
the cubature of Lemma~\ref{lem:SXind41}~(ii),
which is, by Theorem~\ref{thm:Main2}, reduced to
many rational representations involving
much fewer number of fourth powers.

We give one more interesting Hilbert identity,
though it is not always rational.
\begin{theorem}
\label{thm:NS4}
\begin{align}
(\sum_{i=1}^4 x_i^2)^5 &
= \tfrac{1}{2520} \sum_4 (2 x_i)^{10}
+ \tfrac{1}{2520} \sum_8 (x_1 \pm x_2 \pm x_3 \pm x_4)^{10} \nonumber \\
& \qquad + \tfrac{1-120a}{272160}
\sum_{32} (3 x_i \pm x_j \pm x_k \pm x_l)^{10}
+ \tfrac{1-120a}{272160}
\sum_{16} (2 x_i \pm 2 x_j \pm 2 x_k)^{10} \nonumber \\
& \qquad + \tfrac{192a-1}{68040} \sum_{48} (2 x_i \pm x_j \pm x_k)^{10}
+ \tfrac{12-960a}{630} \sum_{12} (x_i \pm x_j)^{10},
\label{eq:NSidentity}
\end{align}
where $\frac{1}{192} \le a \le \frac{1}{120}$.
In particular, if $a$ is rational, then so is
the corresponding identity.
\end{theorem}
\begin{proof} 
The cubature of Theorem~\ref{thm:NS2}~(1) is centrally symmetric,
which is reduced to the half-size formula of index $10$.
The result then follows by (\ref{eq:identity}) and (\ref{eq:identity2}).
\end{proof}

Identity (\ref{eq:NSidentity}) unifies
the following familiar identity by I.~Schur (cf.~\cite[p.~721]{D23}).
\begin{corollary}
\begin{align}
22680 (\sum_{i=1}^4 x_i^2)^5 &
= 9 \sum_4 (2 x_i)^{10} + 9 \sum_8 (x_1 \pm x_2 \pm x_3 \pm x_4)^{10} \nonumber \\
& \qquad + \sum_{48} (2 x_i \pm x_j \pm x_k)^{10}
+ 180 \sum_{12} (x_i \pm x_j)^{10}.
\end{align}
\end{corollary}
\begin{proof} 
Take $a = 1/120$ in (\ref{eq:NSidentity}).
\end{proof}

\begin{remark}
Some classical identities as such by
Lucus (1876) and Liouville (1859),
are often picked up for an introduction in the study of Hilbert identities~\cite{D23}.
It is well known (see, e.g.,~\cite{HP05,R92}) that
Liouville's and Lucas's identities are
closely related by a linear change and
provide essentially the same cubature on $S^3$.
The Hurwitz identity
\begin{align*}
5040 (\sum_{i=1}^4 x_i^2)^4
& = 6 \sum_4 (2 x_i)^8
+ 6 \sum_8 (x_1 \pm x_2 \pm x_3 \pm x_4)^8 \nonumber \\
& \qquad + \sum_{48} (2 x_i \pm x_j \pm x_k)^8
+ 60 \sum_{12} (x_i \pm x_j)^8
\end{align*}
is also well known~\cite[p.~721]{D23}.
It is interesting to note that
Hurwitz's and Schur's identities
are the same in terms of spherical cubature,
i.e., the corresponding formulae have the same weights and
points.
In~\cite{HP05,R92},
this observation is not remarked, though
the relation between Liouville's and Lucas's identities is mentioned.
\end{remark}

The story so far implies how powerful
the cubature approach is to construct Hilbert identities.
In turn, we look at an advantage of
translating spherical cubature into Hilbert identities.

\begin{theorem}
\label{thm:limit}
Let $m \ge 2$ be an integer.
Then
$(\sum_{i=1}^m x_i^2)^4$ does not have
a representation as
an $\mathbb{R}$-linear combination of $(a_1 x_1 + \cdots + a_m x_m)^8$ with
$a_i \in \{0, -1, 1\}$.
\end{theorem}
\begin{proof}
The ratio of the coefficients of $x_1^6 x_2^2$ and $x_1^4 x_2^4$ is $(2:3)$
in $(\sum_{i=1}^n x_i^2)^4$.
But it is $(2:5)$ in any form $(a_1 x_1 + \cdots + a_n x_n)^8$
with $a_i \in \{0, \pm 1\}, 0 \notin \{a_1, a_2\}$.
\end{proof}

\begin{corollary}
\label{cor:limit}
Let $m \ge 2$ and $G$ be a subgroup of $B_m$.
Then
there exists no $G$-invariant Euclidean $8$-design of $\mathbb{R}^m$
that consists of the orbits of the form $(1, \ldots, 1, 0, \ldots, 0)^G$.
\end{corollary}
\begin{proof}
Restricting (\ref{eq:DefEuc}) to homogeneous polynomials of degree $8$
implies the existence of
a cubature of index $8$ on $S^{m-1}$,
by suitably rescaling points and weights.
The result then follows by Theorem~\ref{thm:limit}.
\end{proof}

A variation of Corollary~\ref{cor:limit}
holds for all irreducible reflection groups.
Namely,
Theorem~\ref{thm:NS} can be proved
even if each irreducible reflection group is replaced by its subgroups.

\begin{remark}
(i) 
Corollary~\ref{cor:limit} is the Bajnok theorem for $G = B_m$, and
the case (3) of Theorem~\ref{thm:NS} for $G = D_m$.
It is also interesting to note that
Theorem~\ref{thm:limit} states that
the Bajnok theorem is valid even if
negative coefficients are allowed.
(ii) 
To prove Theorem~\ref{thm:Bajnok},
Bajnok implicitly used the Sobolev theorem.
The approach based on the Sobolev theorem is of theoretic interest,
but it basically requires tedious calculations on
invariant harmonic homogeneous polynomials.
In summary,
the original proof of Bajnok
requires a few pages~\cite[Section 2 and Proposition 15]{B07}
and seems to be involved.
Whereas,
the present proof is short, and simple for
it uses only elementary counting techniques.
The Bajnok theorem is well known in algebra and combinatorics,
however, is not fully recognized in numerical analysis,
though it can be used to determine the maximum degree of
a symmetric cubature on the simplex~\cite{X00-3}
which is traditionally studied in the context of numerical analysis
\footnote{
The second author learned this fact form Yuan Xu.
In~\cite{SX11},
we proved a variation of the Bajnok theorem for
cubature formulae on the simplex,
particularly intended for researchers in numerical analysis.}.
The authors expect that the new proof
will make researchers in many fields
more familiar with the Bajnok theorem.
\end{remark}

\bigskip
\noindent
{\bf Acknowledgement}
This work started when the second author stayed at the Department of Mathematics of
the University of Oregon from April to June in 2011.
He gratefully acknowledges the hospitality of this institution
and the cooporation with Yuan Xu and many other staffs.
The authors also thank Eiichi Bannai, Reinhard Laue, Sanpei Kageyama,
and Oksana Shatalov
for fruitful discussions about regular $t$-wise balanced designs
and index-type cubature.
After an earlier version of this paper was written,
the second author emailed
Bruce Reznick and Koichi Kawada
to discuss the contents of Sections 5 and 6.
They were really patient in
giving us some elementary courses in the subject
and many valuable comments and suggestions,
and the resulting revision extensively improved the previous version.

\appendix 
\section{Classification of $E_8$-invariant cubature}
An $E_8$-invariant cubature of degree $15$
that consists of the orbits of the corner vectors
is classified by the following $27$ types of weights: 
\begin{align*}
&{w_1}= \tfrac{23}{504000}-\tfrac{4288512 {w_2}}{823543}-\tfrac{258048 {w_3}}{15625}-\tfrac{70224 {w_4}}{15625}-\tfrac{15 {w_8}}{128},\\
&{w_5}= \tfrac{3}{224000}-\tfrac{1244160 {w_2}}{823543}-\tfrac{171008 {w_3}}{15625}-\tfrac{79704 {w_4}}{15625}-\tfrac{243 {w_8}}{512},\\
&{w_6}= \tfrac{9}{896000}+\tfrac{4193208 {w_2}}{823543}+\tfrac{507384 {w_3}}{15625}+\tfrac{180792 {w_4}}{15625}+\tfrac{3645 {w_8}}{2048},\\
&{w_7}= \tfrac{67}{672000}-\tfrac{2465280 {w_2}}{823543}-\tfrac{290304 {w_3}}{15625}-\tfrac{94752 {w_4}}{15625}-\tfrac{603 {w_8}}{512},
\end{align*}
and 
\begin{enumerate}
\item ${w_4}=0$, $0\leq {w_2}\leq \tfrac{12588443}{1449551462400}$, $0\leq {w_3}<\tfrac{44118375-4976640000000 {w_2}}{36053104984064}$,\\$0\leq {w_8}\leq \tfrac{88236750-9953280000000 {w_2}-72106209968128 {w_3}}{3126889828125}$,
\item ${w_4}=0$, $0\leq {w_2}\leq \tfrac{12588443}{1449551462400}$, ${w_3}=\tfrac{44118375-4976640000000 {w_2}}{36053104984064}$, ${w_8}=0$,
\item ${w_4}=0$, $\tfrac{12588443}{1449551462400}<{w_2}<\tfrac{117649}{13436928000}$, $0\leq {w_3}\leq \tfrac{44118375-5038848000000 {w_2}}{14396954722304}$,\\$0\leq {w_8}\leq \tfrac{88236750-9953280000000 {w_2}-72106209968128 {w_3}}{3126889828125}$,
\item ${w_4}=0$, $\tfrac{12588443}{1449551462400}<{w_2}<\tfrac{117649}{13436928000}$,\\$\tfrac{44118375-5038848000000 {w_2}}{14396954722304}<{w_3}<\tfrac{338240875-38596608000000 {w_2}}{122407847460864}$,\\$0\leq {w_8}\leq \tfrac{676481750-77193216000000 {w_2}-244815694921728 {w_3}}{1737161015625}$,
\item ${w_4}=0$, $\tfrac{12588443}{1449551462400}<{w_2}<\tfrac{117649}{13436928000}$,\\${w_3}=\tfrac{338240875-38596608000000 {w_2}}{122407847460864}$, ${w_8}=0$,
\item ${w_4}=0$, ${w_2}=\tfrac{117649}{13436928000}$, ${w_3}=0$, $0\leq {w_8}\leq \tfrac{2}{5740875}$,
\item ${w_4}=0$, ${w_2}=\tfrac{117649}{13436928000}$, $0<{w_3}<\tfrac{625}{252829237248}$, $0\leq {w_8}\leq \tfrac{1250-505658474496 {w_3}}{3588046875}$,
\item ${w_4}=0$, ${w_2}=\tfrac{117649}{13436928000}$, ${w_3}=\tfrac{625}{252829237248}$, ${w_8}=0$,
\item ${w_4}=0$, $\tfrac{117649}{13436928000}<{w_2}<\tfrac{2705927}{308772864000}$, $0\leq {w_3}<\tfrac{338240875-38596608000000 {w_2}}{122407847460864}$, \\$0\leq {w_8}\leq \tfrac{676481750-77193216000000 {w_2}-244815694921728 {w_3}}{1737161015625}$,
\item ${w_4}=0$, $\tfrac{117649}{13436928000}<{w_2}<\tfrac{2705927}{308772864000}$, ${w_3}=\tfrac{338240875-38596608000000 {w_2}}{122407847460864}$, ${w_8}=0$,
\item ${w_4}=0$, ${w_2}=\tfrac{2705927}{308772864000}$, ${w_3}=0$, ${w_8}=0$,
\item $0<{w_4}<\tfrac{125}{3092173056}$, $0\leq {w_2}<\tfrac{62942215+7929414597888 {w_4}}{7247757312000}$,\\$0\leq {w_3}<\tfrac{44118375-4976640000000 {w_2}-16803755845632 {w_4}}{36053104984064}$,\\$0\leq {w_8}\leq \tfrac{88236750-9953280000000 {w_2}-72106209968128 {w_3}-33607511691264 {w_4}}{3126889828125}$,
\item $0<{w_4}<\tfrac{125}{3092173056}$, $0\leq {w_2}<\tfrac{62942215+7929414597888 {w_4}}{7247757312000}$,\\${w_3}=\tfrac{44118375-4976640000000 {w_2}-16803755845632 {w_4}}{36053104984064}$, ${w_8}=0$,
\item $0<{w_4}<\tfrac{125}{3092173056}$, ${w_2}=\tfrac{62942215+7929414597888 {w_4}}{7247757312000}$,\\$0\leq {w_3}<\tfrac{338240875-38596608000000 {w_2}-33311510572032 {w_4}}{122407847460864}$,\\$0\leq {w_8}\leq \tfrac{88236750-9953280000000 {w_2}-72106209968128 {w_3}-33607511691264 {w_4}}{3126889828125}$,
\item $0<{w_4}<\tfrac{125}{3092173056}$, ${w_2}=\tfrac{62942215+7929414597888 {w_4}}{7247757312000}$,\\${w_3}=\tfrac{338240875-38596608000000 {w_2}-33311510572032 {w_4}}{122407847460864}$, ${w_8}=0$,
\item $0<{w_4}<\tfrac{125}{3092173056}$, $\tfrac{62942215+7929414597888 {w_4}}{7247757312000}<{w_2}<\tfrac{14706125-1123879249584 {w_4}}{1679616000000}$, $0\leq {w_3}\leq \tfrac{44118375-5038848000000 {w_2}-3371637748752 {w_4}}{14396954722304}$,\\$0\leq {w_8}\leq \tfrac{88236750-9953280000000 {w_2}-72106209968128 {w_3}-33607511691264 {w_4}}{3126889828125}$,
\item $0<{w_4}<\tfrac{125}{3092173056}$, $\tfrac{62942215+7929414597888 {w_4}}{7247757312000}<{w_2}<\tfrac{14706125-1123879249584 {w_4}}{1679616000000}$,\\
$\tfrac{375}{122372096}-\tfrac{307546875 {w_2}}{878720381}-\tfrac{255879 {w_4}}{1092608}<{w_3}<\tfrac{2875}{1040449536}-\tfrac{5453125 {w_2}}{17294403}-\tfrac{209 {w_4}}{768}$,\\
$0\leq {w_8}\leq \tfrac{676481750-77193216000000 {w_2}-244815694921728 {w_3}-66623021144064 {w_4}}{1737161015625}$,
\item $0<{w_4}<\tfrac{125}{3092173056}$, $\tfrac{62942215+7929414597888 {w_4}}{7247757312000}<{w_2}<\tfrac{14706125-1123879249584 {w_4}}{1679616000000}$,\\${w_3}=\tfrac{338240875-38596608000000 {w_2}-33311510572032 {w_4}}{122407847460864}$, ${w_8}=0$,
\item $0<{w_4}<\tfrac{125}{3092173056}$, ${w_2}=\tfrac{14706125-1123879249584 {w_4}}{1679616000000}$, ${w_3}=0$,\\$0\leq {w_8}\leq \tfrac{88236750-9953280000000 {w_2}-33607511691264 {w_4}}{3126889828125}$,
\item $0<{w_4}<\tfrac{125}{3092173056}$, ${w_2}=\tfrac{14706125-1123879249584 {w_4}}{1679616000000}$,\\$0<{w_3}<\tfrac{338240875-38596608000000 {w_2}-33311510572032 {w_4}}{122407847460864}$,\\$0\leq {w_8}\leq \tfrac{676481750-77193216000000 {w_2}-244815694921728 {w_3}-66623021144064 {w_4}}{1737161015625}$,
\item $0<{w_4}<\tfrac{125}{3092173056}$, ${w_2}=\tfrac{14706125-1123879249584 {w_4}}{1679616000000}$,\\${w_3}=\tfrac{338240875-38596608000000 {w_2}-33311510572032 {w_4}}{122407847460864}$, ${w_8}=0$,
\item $0<{w_4}<\tfrac{125}{3092173056}$, $\tfrac{14706125-1123879249584 {w_4}}{1679616000000}<{w_2}<\tfrac{338240875-33311510572032 {w_4}}{38596608000000}$,\\$0\leq {w_3}<\tfrac{338240875-38596608000000 {w_2}-33311510572032 {w_4}}{122407847460864}$,\\$0\leq {w_8}\leq \tfrac{676481750-77193216000000 {w_2}-244815694921728 {w_3}-66623021144064 {w_4}}{1737161015625}$,
\item $0<{w_4}<\tfrac{125}{3092173056}$, $\tfrac{14706125-1123879249584 {w_4}}{1679616000000}<{w_2}<\tfrac{338240875-33311510572032 {w_4}}{38596608000000}$,\\${w_3}=\tfrac{338240875-38596608000000 {w_2}-33311510572032 {w_4}}{122407847460864}$, ${w_8}=0$,
\item $0<{w_4}<\tfrac{125}{3092173056}$, ${w_2}=\tfrac{338240875-33311510572032 {w_4}}{38596608000000}$, ${w_3}=0$, ${w_8}=0$,
\item $\tfrac{125}{3092173056}\leq {w_4}\leq \tfrac{125}{47609856}$, $0\leq {w_2}<\tfrac{14706125-5601251948544 {w_4}}{1658880000000}$,\\$0\leq {w_3}<\tfrac{44118375-4976640000000 {w_2}-16803755845632 {w_4}}{36053104984064}$,\\$0\leq {w_8}\leq \tfrac{88236750-9953280000000 {w_2}-72106209968128 {w_3}-33607511691264 {w_4}}{3126889828125}$,
\item $\tfrac{125}{3092173056}\leq {w_4}\leq \tfrac{125}{47609856}$, $0\leq {w_2}<\tfrac{14706125-5601251948544 {w_4}}{1658880000000}$,\\${w_3}=\tfrac{44118375-4976640000000 {w_2}-16803755845632 {w_4}}{36053104984064}$, ${w_8}=0$,
\item $\tfrac{125}{3092173056}\leq {w_4}\leq \tfrac{125}{47609856}$, ${w_2}=\tfrac{14706125-5601251948544 {w_4}}{1658880000000}$, ${w_3}=0$, ${w_8}=0$.
\end{enumerate}


\begin{thebibliography}{20}
\bibitem{S75}
{\sc G.~N.~Salihov}.
{\em Cubature formulas for a hypersphere that are invariant with respect to the group
of the regular $600$-face}.
Soviet Math. Dokl. {\bf 16} (1975), 1046--1050.
\end{thebibliography}
\end{document}